\documentclass[12pt]{article}
\usepackage{latexsym}
\usepackage{amssymb}
\usepackage{amsmath}
\usepackage{enumerate}
\usepackage{slashed}
\usepackage{mathrsfs}
\usepackage{wasysym}
\usepackage{rawfonts}
\input{prepictex}
\input{pictex}
\input{postpictex}
\usepackage[OT2,OT1]{fontenc}
\def\cyr{%
\renewcommand\rmdefault{wncyr}%
\renewcommand\sfdefault{wncyss}%
\renewcommand\encodingdefault{OT2}%
\normalfont
\selectfont}
\DeclareMathAlphabet{\zap}{OT1}{pzc}{m}{it}
\DeclareTextFontCommand{\textcyr}{\cyr}
\def\be{\begin{equation}}
\def\ee{\end{equation}}
\def\bea{\begin{eqnarray*}}
\def\eea{\end{eqnarray*}}

\def\drc{\slashed{D}}

\def\CC{\mathbb C}

\def\QQ{\mathbb Q}

\newtheorem{main}{Theorem}
\DeclareMathOperator{\End}{End}
\DeclareMathOperator{\Hom}{Hom}
\def\corb{c_1^{orb}}

\DeclareMathOperator{\Ind}{Ind}

\DeclareMathOperator{\tr}{tr}

\newtheorem{thm}{Theorem}[section]
\newtheorem{lem}[thm]{Lemma}
\newtheorem{prop}[thm]{Proposition}
\newtheorem{cor}[thm]{Corollary}
\newtheorem{defn}{Definition}[section]

\newenvironment{proof}{\medskip \noindent
{\bf Proof.}}{\hfill \raisebox{-.2em}{\rule{.7em}{.8em}}
\\}
\newenvironment{xpl}{\mbox{ }\\ {\bf  Example}\mbox{ }}{
\hfill $\diamondsuit$\mbox{}\bigskip}

\def\ZZ{{\mathbb Z}}
\def\RR{{\mathbb R}}
\def\CP{{\mathbb C \mathbb P}}
\begin{document}

\title{Edges, Orbifolds,  and Seiberg-Witten Theory}

\author{Claude LeBrun\thanks{Research supported in part
by NSF grant DMS-1205953.}}

\date{May 7, 2013} 
  
  \maketitle

 \begin{abstract}	
 Seiberg-Witten theory is used to obtain new obstructions to the existence of 
 Einstein metrics on $4$-manifolds with conical singularities along an embedded surface. 
 In the present article, the cone angle is required to be of the form $2\pi /p$, $p$ 
 a positive integer, but we conjecture that similar results  will also hold in greater generality. 
  \end{abstract}

Recent work on  K\"ahler-Einstein metrics by Chen, Donaldson, Sun,   and others \cite{berman,brendedge,cds0,donedge,jmredge}   
  has elicited wide interest in  the existence 
and uniqueness problems for Einstein metrics with conical singularities along a submanifold of real codimension 2. 
This article will show that Seiberg-Witten theory gives rise to  interesting   
 obstructions to the existence of  $4$-dimensional Einstein metrics with conical singularities
   along a surface.  These results are intimately tied to known phenomena in K\"ahler geometry, and 
        reinforce the overarching  principle  that K\"ahler  metrics 
play a uniquely  privileged role   in $4$-dimensional Riemannian geometry, to a degree  that is  simply unparalleled in other  dimensions.

%For technical reasons, we are currently only able to prove these results
%   for rather special values of the cone angle, but these results do suggest some obvious conjectures
%   that ought to hold in much greater generality.
%   
We  now   recall the definition \cite{atleb} of an edge-cone metric  on a $4$-manifold. 
Let $M$ be a smooth  compact $4$-manifold,  let $\Sigma \subset M$ be a smoothly
embedded  compact surface.
Near any point $p\in \Sigma$, we can thus
find local coordinates $(x^1,x^2,y^1,y^2)$ in which $\Sigma$
is given by $y^1=y^2=0$. Given any such adapted 
coordinate system, we then introduce an associated  
{\em transversal polar coordinate}  system  $(\rho,\theta, x^1, x^2)$
 by setting  $y^1=\rho \cos \theta$ and  $y^2= \rho \sin \theta$.
 Now fix some positive constant $\beta > 0$. 
 An edge-cone metric $g$ of cone angle
$2\pi\beta$ on $(M,\Sigma)$ is a smooth Riemannian metric  on $M-\Sigma$ which 
takes the form 
\begin{equation}
\label{edge}
g= d\rho^2 + \beta^2 \rho^2 (d\theta+ u_jdx^j)^2 + w_{jk} dx^jdx^k  + \rho^{1+\varepsilon} h
\end{equation}
in  a suitable  transversal polar coordinate  system near each point of  $\Sigma$, where
the  symmetric tensor field $h$ on $M$ is required to have {\em infinite conormal regularity} 
along $\Sigma$. This last assumption  means that the components of $h$ in $(x^1,x^2,y^1,y^2)$  coordinates have continuous
derivatives of all orders with respect to collections of smooth  vector fields (e.g.  $\rho ~\partial /\partial \rho$, $\partial /\partial \theta$, 
$\partial/\partial x^1$,  $\partial/\partial x^2$) which have vanishing normal component along $\Sigma$. 
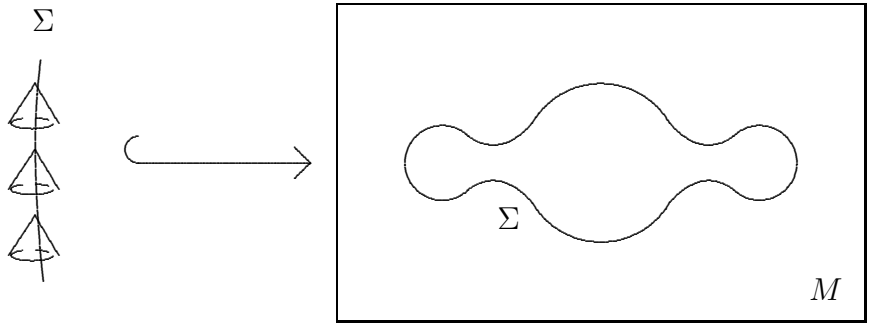
\begin{center}
\mbox{
\beginpicture
\setplotarea x from 30 to 400, y from -10 to 110
{\setquadratic
\plot 92 89  90  52  93 5 /
}
\putrectangle corners at 200 110 and 400 -10
\circulararc 270 degrees from 250 60  center at 240 50
\circulararc -270 degrees from 350 60 center at 360 50
\circulararc 120 degrees from 326 65 center at 300 50 
\circulararc -120 degrees from 326 35  center at 300 50 
\put {$M$} [B1] at 385 -2
\put {$\Sigma$} [B1] at 265 25
\put {$\Sigma$} [B1] at 89 100
\arrow <6pt> [1,2] from 125 50  to 190 50
\circulararc  -180 degrees from 125 50  center at 125 55
{\setquadratic
\plot 350 60 338 57 326 65 /
\plot 350 40 338 43 326 35 /
\plot 250 40 262 43 274 35 /
\plot 250 60 262 57 274 65 /
} 
{\setlinear 
\plot  72 15  82 30   /
\plot   91 15  82 30  /
\plot  72 65  82 80   /
\plot   91 65  82 80  /
\plot  72 40  82 55   /
\plot   91 40  82 55  /
}
\ellipticalarc axes ratio 4:1 180 degrees from  69 15
center at  77 15
\ellipticalarc axes ratio 4:1 180 degrees from  69 65
center at  77 65
\ellipticalarc axes ratio 4:1 180 degrees from  69 40
center at  77 40
\setdashes 
\ellipticalarc axes ratio 4:1 -180 degrees from 69 15
center at 77 15
\ellipticalarc axes ratio 4:1 -180 degrees from 69 65
center at 77 65
\ellipticalarc axes ratio 4:1 -180 degrees from 69 40
center at 77 40
\endpicture
}
\end{center}
Thus, an edge-cone metric $g$ behaves like
a smooth metric in directions parallel to $\Sigma$, but is modelled on 
a $2$-dimensional cone
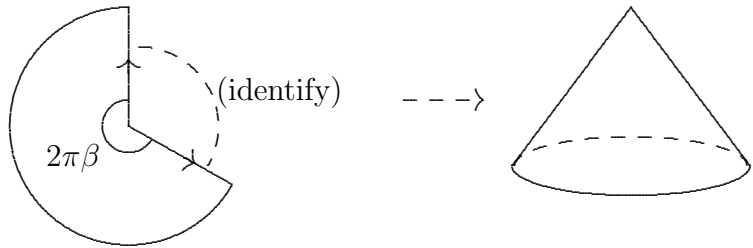
\begin{center}
\mbox{
\beginpicture
\setplotarea x from 0 to 270, y from 55 to 155
\circulararc 240 degrees from 30 145 center at 30 100 
\circulararc 240 degrees from 30 110 center at 30 100 
\ellipticalarc axes ratio 4:1 180 degrees from 175 85
center at 220 85
{\setlinear 
\plot  30 100 30 145 /
\plot  30 100 69 78  /
\plot  220 145 175 85   /
\plot  220 145 265 85   /
}
\put {$2\pi \beta$} [B1] at 5 85
\put {(identify)} [B1] at 83 110
\arrow <3pt> [1,3] from 26 120  to 26 125
\arrow <3pt> [1,3] from 47 88  to 51  86
\setdashes 
\arrow <3pt> [1,3] from 130 110  to 160 110
\circulararc -120 degrees from 30 130 center at 30 100
\ellipticalarc axes ratio 4:1 -180 degrees from 172 85
center at 216 85
\endpicture
}
\end{center}
in the transverse directions.

An edge-cone  metric $g$ is said to be {\em Einstein} if its
Ricci  tensor $r$ 
  satisfies %the Einstein equation
\begin{equation}
\label{einstein}
r= \lambda g
\end{equation}
on $M-\Sigma$,  where 
 $\lambda$ is an undetermined real constant,  
  called the Einstein constant of $g$.
In other words, an edge-cone metric is Einstein iff 
 it has constant Ricci curvature on the complement of $\Sigma$.

Many  interesting edge-cone metrics  arise as orbifolds metrics. 
Suppose that $M$ is a smooth  oriented $4$-manifold and that $\Sigma\subset M$
is a smooth  oriented surface. By the tubular neighborhood theorem, $\Sigma$ therefore
has a neighborhood which is diffeomorphic to the unit disk bundle in a complex line bundle 
$L\to \Sigma$. Choosing some conformal structure on $\Sigma$, we can then introduce local 
complex coordinates $(w,z)\in \CC^2$ near any point of $\Sigma$, so that $z=0$ is a local defining function for 
$\Sigma$, and such that $|z|< 1$ corresponds to the given tubular neighborhood. Now choose
a natural number $p\geq 2$, and set $\beta = 1/p$. If we may then introduce an auxiliary 
complex coordinate $\zeta$ by declaring\footnote{Here we have divided by $|\zeta|^{p-1}$ in order to ensure that $|z|=|\zeta|$. 
Later, however,  it will 
 sometimes be useful to ignore this factor,  instead viewing it as representing a fixed self-homeomorphism
of $M$ that stretches our tubular neighborhood  radially, away from $\Sigma$.}
 that $z= \zeta^p/|\zeta|^{p-1}$. 
This allows us to 
think of $\CC^2$ as $\CC^2/\ZZ_p$, where the action of $\ZZ_p$ on $\CC^2$ is generated by 
$(w,\zeta)\mapsto (w, e^{2\pi i/p}\zeta)$, and we can thus equip $M$ with an orbifold structure
by supplementing out original smooth atlas on $M-\Sigma$ with multi-valued coordinates 
$(x^1, x^2, x^3, x^4)$ near points of $\Sigma$, where $(w, \zeta)=(x^1+ix^2, x^3+ix^4)$.  
Following standard practice in discussing orbifolds \cite{bg,satake}, 
we will use the term {\em local uniformizing chart} to  refer to either a smooth chart 
around a point of $M-\Sigma$, or to one of these  multi-valued coordinate systems on 
a neighborhood of a point of $\Sigma$; the  {\em local uniformizing group} is then 
the trivial group $\{ 1\}$ for a standard smooth chart, or $\ZZ_p$ for a multi-valued chart near a point of $\Sigma$. 
Modulo self-diffeomorphisms of the pair 
$(M, \Sigma)$,
the orbifold  constructed in this way is then independent of choices,   and will be denoted by  $(M,\Sigma, \beta )$, where  $\beta = 1/p$.
An orbifold metric $g$ on $(M,\Sigma, \beta )$  is then an object which is locally represented
near $\Sigma$ as a $\ZZ_p$-invariant Riemannian metric in local uniformizing coordinates, and 
 which, on $M-\Sigma$,  is
a smooth Riemanian metric  in the usual sense. 
An orbifold
metric is said to be Einstein iff equation \eqref{einstein} holds in a locally uniformizing chart near each point. 

If   $\beta  = 1/p$
for some positive integer $p$, 
every orbifold metric on $(M, \Sigma, \beta )$
may  be viewed as an edge-cone metric  on $(M, \Sigma)$
of  cone angle $2\pi\beta$. 
Of course,  the converse is generally false;    our definition allows edge-cone metrics   to have so  little regularity
that   most of them  do not even have bounded curvature as one approaches $\Sigma$.  
However, the  converse {\em does} hold in the Einstein case: 
Proposition \ref{reggie}, proved in \S \ref{regal} below,   shows that,  when  $\beta  = 1/p$, every 
edge-cone
Einstein  metric of cone angle $2\pi \beta$   is simply an orbifold Einstein metric  described  from an alternate perspective. This fact
will allow us to obtain results concerning Einstein-edge cone metrics with $\beta  = 1/p$ by applying ideas from the theory of  orbifolds.

 In the author's  joint paper with Atiyah \cite{atleb}, topological 
  obstructions to the existence of  Einstein edge-cone metrics  were obtained. These  obstructions, 
  which generalize the
 Hitchin-Thorpe inequality \cite{bes,hit,tho} for non-singular Einstein metrics, 
 are stated purely in terms of    homeomorphism invariants of 
 the pair $(M,\Sigma)$. 
 While these obstructions do reflect   several peculiar features of $4$-dimensional Riemannian geometry,
there is every reason to believe that they  barely  scratch the surface of what must actually  be true. For example, in the non-singular
 case, there are also obstructions \cite{lmo,lno,lric} to the existence of Einstein  metrics which 
 depend on {\em diffeomorphism} invariants rather than on homeomorphism invariants. 
 The purpose of this paper is to  explore the manner in which these results, which are proved using Seiberg-Witten theory, 
 generalize to the edge-cone setting. Our main result is the following:

  \begin{main} \label{non}
  Let $X$ be a smooth compact oriented $4$-manifold, and let $\Sigma\subset X$ be a smooth compact
  oriented embedded surface. Fix some integer $p\geq 2$, and set $\beta = 1/p$. Suppose that 
  $X$ admits a symplectic form $\omega_0$  for which $\Sigma$ is a symplectic submanifold, and such that 
   \begin{equation}
  \label{common}
(c_1 (X)+ (\beta -1) [\Sigma ])\cdot [\omega_0 ] <  0
\end{equation}
where  $[\omega_0 ], c_1(X),  [\Sigma] \in H^2(X)$ are respectively the deRham class of $\omega$, 
 the first Chern class  of $(X, \omega)_0$, 
 and the Poincar\'e dual of  $\Sigma$. Choose 
a non-negative  integer  $\ell$   such that 
   \begin{equation}
\label{blownup}
\ell \geq \frac{1}{3} (c_1(X) + (\beta -1)[\Sigma ])^2,
\end{equation}
 let $M \approx X\#\ell \overline{\CP}_2$ 
  be the manifold obtained by blowing $X$  up at $\ell$ points of $X-\Sigma$,
   and notice that $\Sigma$ can  also be viewed as a submanifold of $M$. 
 Then $(M,\Sigma)$ does not carry any   Einstein  edge-cone metrics of cone angle $2\pi\beta$.
     \end{main}
   
 It should be emphasized  that the symplectic form $\omega_0$  plays a purely auxiliary role here,   and merely guarantees that a certain  
  differential-topological invariant is non-trivial;  the putative 
  Einstein metrics under discussion here are emphatically {\em not} assumed to satisfy any local condition involving $\omega_0$. 
        The constraint $\beta  = 1/p$ on the cone angle is essential for our proof, but is presumably just an artifact of our method. 
 In any case,   we will see  in \S \ref{concrete} below that 
Theorem \ref{non}   obstructs the existence of orbifold Einstein metrics in
  concrete circumstances where the  results of \cite{atleb}
do not lead to such a conclusion.

 Notice that, since $\beta = 1/p  < 1$ and 
 $\int_\Sigma \omega_0 > 0$,  condition  \eqref{common} automatically holds whenever $c_1(X) \cdot [\omega_0 ] \leq 0$. 
 Results of Taubes \cite{taubes2} and Liu \cite{liu1}  therefore imply  that \eqref{common} automatically holds whenever   $X$ 
is not  a blow-up of  $\CP_2$ or some   ruled surface.  But even  in these exceptional cases, 
  \eqref{common} will still hold whenever   $\Sigma \subset X$ has sufficiently high degree with respect to $[\omega_0 ]$.

  Our proof of Theorem \ref{non}  uses  Proposition \ref{reggie}  to first  reduce
  the problem to a question regarding orbifold  metrics. An orbifold version of Seiberg-Witten theory 
  is then used to obtain  curvature estimates  without imposing the Einstein condition, but assuming
  orbifold regularity. Our main result in this direction is the following:

  \begin{main} \label{oui}
  Let $(X, \omega_0)$ be a compact  symplectic $4$-manifold, 
  and let $\Sigma \subset X$
  be a compact embedded symplectic surface. Let $M \approx X\#\ell \overline{\CP}_2$ 
  be the manifold obtained by blowing up $X$ at $\ell\geq 0$ points that do not belong to $\Sigma$. 
  For a positive integer $p$, 
  set $\beta = 1/p$, and let $(M,\Sigma,\beta )$ 
  be the $4$-dimensional orbifold obtained from $M$ by declaring the total angle around
  $\Sigma$ to be $2\pi\beta$. If 
   \eqref{common} holds, 
  then the curvature of any orbifold Riemannian metric $g$ on $(M,\Sigma,\beta )$ satisfies 
  \begin{eqnarray}
  \label{scalar}
\int_M s^2 d\mu &\geq& 32\pi^2 (c_1 (X)+ (\beta -1) [\Sigma ])^2\\
\label{mixed}
\int_M (s-\sqrt{6}|W_+|)^2 d\mu &\geq & 72\pi^2 (c_1 (X)+ (\beta -1)[\Sigma])^2
\end{eqnarray}
where $s$, $W_+$ and $d\mu$ respectively denote the scalar curvature, self-dual Weyl curvature,
and volume form of $g$. 
Moreover, both inequalities are strict  unless $\ell=0$ and 
the orbifold metric $g$ on $(X,\Sigma,p)$
is K\"ahler-Einstein.
  \end{main}
  
  One might hope that such curvature estimates will also  hold  for edge-cone metrics,
  and not just for metrics with orbifold regularity. At the moment, however, this is far from obvious, 
  even when $\beta = 1/p$.   Nonetheless, it seems reasonable to hope that results like Theorem \ref{non}
  will in fact also hold for more general cone angles. 
    For the present, however,  we believe that it is useful   to explain what can currently be proved for these special cone angles,
 in the hope  that these results  will 
 provide a reliable  indicator  of what  one might   expect  to hold in greater generality.

\section{Edges and Orbifolds}
\label{regal}
 
Let us  begin our discussion by proving a   regularity result,  inspired by  \cite{cds3},
  which will play a key  role  in what follows. 
 The gist is that,  when $\beta = 1/p$, every  Einstein 
 edge-cone metric is actually an  orbifold metric.

\begin{prop} 
\label{reggie}
Let $g$ be an Einstein edge-cone metric on $(M,\Sigma)$ of cone angle $2\pi \beta$, 
where $\beta = 1/p$ for some positive integer $p$. Then $g$ naturally extends  to $(M, \Sigma, \beta)$
as an orbifold Einstein metric. 
\end{prop}
\begin{proof} After passing to the ramified cover obtained by introducing the new angular coordinate $\tilde{\theta}= \beta\theta$
and then  setting $x^3=\rho\cos \tilde{\theta}$ and $x^4= \rho \sin \tilde{\theta}$,  
our edge-cone metric \eqref{edge} takes the form 
$$g =  g_0 + \rho^{1+\varepsilon} h$$
where $g_0$ is a smooth metric in $(x^1, x^2, x^3, x^4)$ coordinates, and where the tensor field $h$ has infinite conormal regularity along 
the $x^1x^2$ coordinate plane. 
Now notice that  
\begin{eqnarray*}
\frac{\partial}{\partial x^3} \rho^{1+\varepsilon} h_{jk} &=& (\cos \tilde{\theta})  \frac{\partial}{\partial \rho} \rho^{1+\varepsilon} h_{jk} - 
\frac{\sin \tilde{\theta}}{\rho}  \frac{\partial}{\partial \tilde{\theta}} \rho^{1+\varepsilon} h_{jk}\\&=&
\rho^{\varepsilon} \left[ (\cos \tilde{\theta})\left(
\rho  \frac{\partial}{\partial \rho} + (1+\varepsilon) \right) h_{jk} - (\sin \tilde{\theta}) \frac{\partial}{\partial \tilde{\theta}} h_{jk}
\right]
\end{eqnarray*}
is of class $C^{0,\varepsilon}$, and similar computations for the other first partial derivatives of 
$\rho^{1+\varepsilon} h_{jk}$ show that they, too, belong to $C^{0,\varepsilon}$. 
It follows that 
the  components $g_{jk}$ of $g$  in  
$(x^1, x^2, x^3, x^4)$ coordinates are of class $C^{1,\varepsilon}$. 
Similarly,  on a neighborhood of, say, the origin, 
the second derivatives of $g_{jk}$
are bounded by a constant times $\rho^{\varepsilon -1}$.
Since $\rho^{2(\varepsilon -1)}$ has finite integral  on a disk about  the origin in $\RR^2$, 
this means that, on a neighborhood of the origin in $\RR^4$, the metric
components $g_{jk}$ actually belong to $C^{1,\varepsilon}\cap L^2_2$. 

Since $g$ is of class $C^{1,\varepsilon}$, there exist  harmonic coordinates
$(\tilde{x}^1, \tilde{x}^2, \tilde{x}^3, \tilde{x}^4)$,  which  depend on $(x^1, x^2, x^3, x^4)$ via 
a $C^{2, \varepsilon}$ diffeomorphism  \cite[Lemma 1.2]{det-kaz}, and rewriting the metric in these coordinates does not 
alter the fact that the metric components $g_{jk}$ belong to $C^{1,\varepsilon}\cap L^2_2$. 
In these harmonic coordinates, the Einstein equation \eqref{einstein} now takes the form 
\begin{equation}
\label{harmony} 
\Delta g_{jk} = 2 \lambda g_{jk} + Q_{jk} ( g, \partial g)
\end{equation}
and our hypotheses tell us that this equation is satisfied {\em almost everywhere} in the classical sense. Here the metric Laplacian takes the simplified form  
$$
\Delta = - g^{\ell m} \partial_\ell \partial_m
$$
in our harmonic coordinates  \cite{det-kaz}, while 
 the term $Q_{jk}(g , \partial g)$ is quadratic in first derivatives of $g$, with coefficients expressed in terms of 
$g$ and its inverse, and  so, in our case,  is of class $C^{0,\varepsilon}$. 
Thus, for any choice of $j$ and $k$, the function $u= g_{jk} \in L^2_2\cap L^\infty$ solves an  elliptic equation 
$$
a^{\ell m } \partial_\ell \partial_m  u = f
$$
almost everywhere, 
where the coefficients $a^{\ell m}=-g^{\ell m}$ and the inhomogeneous term  $f=-2 \lambda g_{jk} - Q_{jk} ( g, \partial g)$ both belong to  $C^{0,\varepsilon}$. 
By a classic    result of 
Ladyzhenskaya-Ural$'$tseva \cite[Ch. 3, Theorem 12.1]{lady-ural}, it  follows that the function  $u=g_{jk}$
is actually a function of class $C^{2,\varepsilon}$. In particular, by continuity, $g_{jk}$  must solve \eqref{harmony}  everywhere in the classical sense, so that 
$g$ is a bona fide Einstein metric in our entire coordinate domain. Bootstrapping then allows one to 
conclude that $g_{jk}$ is smooth. Indeed, 
a stronger  result of DeTurck-Kazdan \cite[Theorem 5.2]{det-kaz} actually    tells us that $g$ is 
real analytic in these harmonic coordinates; moreover, it is real analytic  in geodesic normal coordinates as well. 

Since the given $\ZZ_p$ action 
preserved our $C^1$ metric $g$ in the original $(x^1, x^2, x^3, x^4)$  coordinate system, it sends geodesics of our real-analytic metric to other geodesics. 
Hence  this finite group of isometries becomes a group of linear maps of $\RR^4$ in geodesic normal coordinates
centered at a  point of $\Sigma$. Modding out by this action makes this chart into a local uniformizing
chart for $(M, \Sigma, \beta)$. Moreover, the transition functions between two such charts preserves
the real-analytic metric $g$, and is therefore real-analytic. In particular, these charts 
are compatible with our original smooth atlas. 
\end{proof}

While the above proof  has, for notational consistency, only been presented here in dimension $4$, 
one can define an edge-cone metric in arbitrary dimension  \cite{atleb}, and 
it is therefore worth pointing out that 
the same argument works without essential changes in dimension $n$. In this setting, one thus  obtains the same 
regularity result for  Einstein edge-cone  metrics $g$
with cone angle $2\pi /p$  along an arbitrary   submanifold 
$\Sigma^{n-2} \subset M^n$ of  codimension $2$.

\section{Orbifolds, Indices, and All That}

In this section, we explain various simple  facts about orbifolds that will turn out to be vital for our purposes. 
  
  Recall \cite{bg,satake} that a smooth $n$-dimensional orbifold $X$ is a Hausdorff second-countable space,
  equipped with an open covering $\{ U_\mathfrak{J} \}$ and a collection of homeomorphisms $\phi_\mathfrak{J} : U_\mathfrak{J} \to V_\mathfrak{J}/\Gamma_\mathfrak{J}$, 
  where $V_\mathfrak{J} \subset \RR^n$ is an open set and   $\Gamma_\mathfrak{J}< GL(n, \RR)$ is a finite matrix group, such that 
  the transition functions 
  $\phi_{\mathfrak{JK}}:=\phi_\mathfrak{J} \circ \phi_\mathfrak{K}^{-1}: U_\mathfrak{K}\cap U_\mathfrak{J} \to U_\mathfrak{J}\cap U_\mathfrak{K}$  lift as diffeomorphisms $\tilde{\phi}_{\mathfrak{JK}}$ between appropriate regions of $\RR^n$. 
  The multi-valued maps $U_\mathfrak{J}\multimap\to V_\mathfrak{J}\subset \RR^n$ induced by the $\phi_\mathfrak{J}$ are called {\em local uniformizing charts}
  for $X$, while the groups $\Gamma_\mathfrak{J}$ are called the associated {\em local uniformizing groups}. 
  Any orbifold can be written as the disjoint 
  union of its 
  {\em regular set} and its {\em singular set}. The singular set consists of points which correspond to  fixed points of $\Gamma_\mathfrak{J}$ in
  some local uniformizing chart; its complement, the regular set, is open and dense, and  our definition canonically endows the
  regular set with the structure of a   smooth $n$-manifold. 

If $M$ is a smooth compact oriented $4$-manifold,  if $\Sigma\subset M$
is a compact oriented surface, and if  $p\geq 2$ is any integer, we have already observed that we may  endow $M$ with 
an orbifold structure by using the usual smooth charts on $M-\Sigma$, while  modeling
neighborhoods of  point of $\Sigma$ on $\RR^4/\ZZ_p$, where the $\ZZ_p$ action
on $\RR^4 = \CC^2$ is generated by $(z_1,z_2) \mapsto (z_1, e^{2\pi i/p}z_2)$. 
Throughout this article,  the  resulting orbifold is denoted by $(M, \Sigma , \beta)$, where $\beta = 1/p$. 
For these examples,  the singular set is of course $\Sigma$, whereas
$M-\Sigma$ is 
the regular set.

\subsection{DeRham Cohomology and Hodge Theory}     
    
  A tensor field $\psi$ on a smooth orbifold $X$ is an object which is represented by a smooth, $\Gamma_\mathfrak{J}$-invariant tensor field on the co-domain  $V_\mathfrak{J}$
  of each local uniformizing chart, such that 
  these local representatives transform under the transition functions $\tilde{\phi}_{\mathfrak{JK}}$ via the appropriate representation of the Jacobian matrix. 
  In other words, a  tensor field on $X$ is a smooth tensor field on the regular set which  extends  across each singular point 
as a $\Gamma_\mathfrak{J}$-invariant tensor field in a local uniformizing chart.  For example, as we have already seen, 
an orbifold metric is an object which, in locally
uniformizing charts,  is simply a Riemannian metric which is invariant under the action of the locally uniformizing group.

Another  important class of tensor fields consists of the skew-symmetric  covariant tensors, better known as  differential forms.
If $X$ is a smooth orbifold, we thus have, 
for each non-negative integer $k$,
  a  sheaf $\mathcal{E}^k$
of real-valued 
differential $k$-forms on  $X$. This sheaf is fine, in the sense that 
it admits partitions of unity, and hence is acyclic. The usual Poincar\'e lemma, 
combined with averaging over the action of $\Gamma_\mathfrak{J}$, shows that the Poincar\'e lemma
also holds for orbifolds. The orbifold deRham complex therefore provides an acyclic
resolution 
$$0\to \RR \to \mathcal{E}^0  \stackrel{d}{\rightarrow} \mathcal{E}^1\stackrel{d}{\rightarrow} \mathcal{E}^2 \stackrel{d}{\rightarrow} \cdots \stackrel{d}{\rightarrow} \mathcal{E}^k   \stackrel{d}{\rightarrow} \cdots$$
of the constant sheaf. The abstract deRham theorem \cite{wells} thus shows that 
the deRham cohomology of any orbifold computes the \v{C}ech cohomology with real coefficients; and since
any orbifold is locally contractible, this \v{C}ech cohomology  is in turn isomorphic to 
the singular cohomology  of the 
of the underlying topological space,  with real coefficients: 
$$H^k(X, \RR) \cong \check{H}^k (X, \RR)\cong H_{DR}^k(X, \RR):= \frac{\ker d: \mathcal{E}^k(X) \to \mathcal{E}^{k+1}(X)}{\mbox{im } d: \mathcal{E}^{k-1}(X) \to \mathcal{E}^{k}(X)} ~.$$
Cohomology with complex coefficients can similarly be computed using 
complex-valued orbifold forms $\mathcal{E}^\bullet_\CC$:
$$H^k(X, \CC) \cong \check{H}^k (X, \CC)\cong  H_{DR}^k(X, \CC):= \frac{\ker d: \mathcal{E}^k_\CC(X) \to \mathcal{E}_\CC^{k+1}(X)}{\mbox{im } d: \mathcal{E}_\CC^{k-1}(X) \to \mathcal{E}_\CC^{k}(X)} ~.$$
Moreover, the cup product of orbifold deRham classes is induced by the wedge product:
$$[\varphi ] \smile [\psi] = [\varphi \wedge \psi ]~.$$
This can be proved by exactly the same argument one  uses in the  non-singular case \cite[Theorem 14.28]{bottu}. 

It makes perfectly good sense to integrate an $n$-form on a smooth compact oriented connected orbifold $X$,
and  integration  gives rise to  an isomorphism $H^n_{DR}(X) \cong \RR$, exactly as in the  manifold case. 
Of course, the fact that integration is  defined as an operation on cohomology depends on the observation that 
Stokes theorem is also valid for orbifolds.

If $X$ is a smooth compact orbifold, equipped with an orbifold Riemannian metric $g$,
then every deRham class is represented by a unique harmonic representative. To make this precise, 
let us assume, for simplicity, that the  $n$-dimensional compact orbifold $X$ is oriented, which is to say that there is a fixed global choice $d\mu$ of
volume $n$-form which is compatible with the given metric $g$.
We can then  define the Hodge star operator
$$\star :  \mathcal{E}^k(X)\to  \mathcal{E}^{n-k}(X)$$
by requiring that 
$$
\varphi \wedge \star \psi = \langle \varphi , \psi \rangle ~d\mu
$$
for any two $k$-forms $\varphi$ and $\psi$, where $d\mu$ is the 
metric orbifold  $n$-volume form and where the point-wise inner product of forms is 
the one induced by $g$. 
The space of (orbifold) harmonic $k$-forms on $X$ is then defined by 
$$
\mathcal{H}^k(X,g) := \left\{ \varphi \in \mathcal{E}^k(X) ~\Big| ~d\varphi= 0, ~ d(\star \varphi) = 0\right\}.
$$
The {\em Hodge theorem} for orbifolds asserts that the tautological map 
\begin{eqnarray*}
\mathcal{H}^k(X,g) &\longrightarrow& H_{DR}^k(X, \RR)\\
\varphi \qquad &\longmapsto& \qquad [\varphi ]
\end{eqnarray*}
is an isomorphism. The injectivity of this map is elementary;  if $\varphi$ is 
harmonic, it is $L^2$-orthogonal to any exact form, so that 
$$
\int_X \|\varphi + d\psi\|^2 d\mu = \int_X \|\varphi \|^2 d\mu+ \int_X \| d\psi \|^2 d\mu~,
$$
making  $\varphi$  the unique minimizer   of  the $L^2$-norm in its deRham class. 

\subsection{Vector $V$-Bundles}

Just as tensor fields on a manifold are sections of appropriate vector bundles, tensor fields on 
orbifolds are sections of vector $V$-bundles \cite{bg,rotho3,satake}. A vector $V$-bundle over an orbifold $X$ has a total space
$E$ which is an orbifold, and a projection $\varpi : E\to X$ which is a smooth submersion of 
orbifolds. However, in contrast to the situation for ordinary vector bundles, we do not require
this projection to be locally trivial. Instead, we merely require that, for some fixed real or complex
vector space $\mathbf{V}$, there is a system of
 local uniformizing charts $U_\mathfrak{J} \to V_\mathfrak{J}/\Gamma_\mathfrak{J}$ for $X$ which is compatible with a 
 a system of locally uniformizing charts $\varpi^{-1} (U_\mathfrak{J})  \to (V_\mathfrak{J}\times \mathbf{V}) /\Gamma_\mathfrak{J}$ for $E$,
where the action of $\Gamma_\mathfrak{J}$ on $V_\mathfrak{J}\times \mathbf{V}$ is the given action on $V_\mathfrak{J}$ times 
some representation $\varrho: \Gamma_\mathfrak{J} \to \End (\mathbf{V})$; as usual, the transition functions
for $E$ are moreover required to lift to  fiber-wise linear maps. Thus, the fiber $\varpi^{-1}(x)$ over any regular point  $x\in X$ will be a copy of
$\mathbf{V}$. However, if $x\in X$ is a singular point, $\varpi^{-1}(x)$ is  in principle merely  the quotient of  $\mathbf{V}$
by a finite group, and so, in particular, may not even be a vector space. For example, the tangent bundle $TX\to X$
is a $V$-bundle, as are the tensor bundles $(\bigotimes^k TX) \otimes (\bigotimes^\ell T^*X)\to X$ and their
complexifications. 

A {\em section} of a $V$-bundle $\varpi : E\to X$  is a right inverse $f: X\to E$ for $\varpi$ which is locally
represented by  an equivariant smooth function $V_\mathfrak{J}\to \mathbf{V}$. For example, a real tensor field is
exactly a section of  one of the $V$-bundles $(\bigotimes^k TX) \otimes (\bigotimes^\ell T^*X)\to X$; a complex tensor
field is a section of the complexification of one of these bundles. The space of sections of a vector $V$-bundle
$E\to X$ will be denoted by $\mathcal{E}(X,E)$; considering sections over arbitrary open subsets of 
$U\subset X$ gives rise to a sheaf of $X$ which will simply be denoted by $\mathcal{E}(E)$. In the same vein, we will 
use $\mathcal{E}^k(E)$ to denote the sheaf of sections of the $V$-bundle $(\Lambda^kT^*X)\otimes_\RR E$. 

It also makes sense to talk about
connections on vector $V$-bundles. By definition, a connection $\nabla$ on a
vector $V$-bundle $E\to M$ is a linear operator
$$
\nabla : \mathcal{E}(E)\to \mathcal{E}^1 (E)
$$
which in local uniformizing charts is represented by a $\Gamma_\mathfrak{J}$-equivariant connection 
on the trivial vector bundle $V_\mathfrak{J}\times \mathbf{V}$. 

We will need to pay especially close attention to the special case of $V$-line bundles,
which are by definition vector $V$-bundles with  generic fiber $\CC$. 
 For such a $V$-bundle $\varpi : L \to X$, there is thus    a system of
 local uniformizing charts $U_\mathfrak{J} \to V_\mathfrak{J}/\Gamma_\mathfrak{J}$ for $X$ which is compatible with a 
 a system of locally uniformizing charts $\varpi^{-1} (U_\mathfrak{J})  \to (V_\mathfrak{J}\times \CC) /\Gamma_\mathfrak{J}$ for $L$,
where the action of $\Gamma_\mathfrak{J}$ on $V_\mathfrak{J}\times \CC$ is the given action on $V_\mathfrak{J}$ times 
some representation $\Gamma_\mathfrak{J} \to U(1)$. 
Since $\varpi : L\to X$ may not be locally trivial, so we cannot immediately invoke standard machinery to 
define its Chern class. However, 
 tensor products of $V$-line  bundles are again $V$-line  bundles; 
 in particular, if $L\to X$ is any $V$-line  bundle, its tensor powers $L^{\otimes q}$ are also $V$-line  bundles. Now assume the orders $|\Gamma_\mathfrak{J}|$
of the local uniformizing groups are bounded --- as is automatically true  if $X$ is compact. If $q$ is the 
least common multiple of $\{ |\Gamma_\mathfrak{J}| \}$, then $L^{\otimes q}$ is  locally trivial, 
and so is a line bundle in the conventional sense. We may then define the {\em orbifold Chern class} of $L$ 
by 
$$\corb (L) = \frac{1}{q} c_1 (L^{\otimes q})\in H^2 (X, \QQ ).$$
If $L$ happens to be a complex line bundle in the conventional sense, this reproduces its  usual rational Chern class.
Of course, this construction ignores the torsion part of the usual Chern class, but this loss of information will not be 
an issue for present purposes.

The orbifold Chern class of an $V$-line  bundle can also be obtained directly, via Chern-Weil theory.
If $\nabla$ is any connection on the $V$-line bundle $L\to X$, its curvature $F_\nabla$ is 
a closed $2$-form on $X$. Moreover, any other connection on $L$ can be expressed as 
$$\tilde{\nabla} = \nabla + {\zap A}$$
for a unique complex-valued $1$-form ${\zap A}$, and 
the new connection's curvature is expressible in terms of the old one by 
$$F_{\tilde{\nabla}}=  F_\nabla+ d{\zap A}.$$
Thus the deRham class of of $F_\nabla$ is independent of the choice of connection, 
and is therefore an invariant of $L$. Indeed, one can show that 
$$\corb (L ) = [\frac{i}{2\pi} F_{\nabla}] \in H^2_{DR}(M , \CC)$$
where we have identified $H^2(X, \QQ)$ with a subset of $H^2(X,\CC) \cong H^2_{DR}(X, \CC)$
in the  usual way. 
Indeed,  curvature forms are  additive under tensor products,  so it suffices to prove the assertion for the case 
when $L$ is locally trivial; and for locally trivial line bundles,  the usual \v{C}ech-deRham proof \cite[\S III.4]{wells} immediately 
generalizes from manifolds to
orbifolds with only cosmetic changes. 

Since the space of connections on a $V$-line bundle is an affine space modeled on 
$\mathcal{E}^1_\CC$, we shall engage in a standard abuse of notation by referring to  a
connection on $L$ as ${\zap A}$ rather than as $\nabla$; in this context, 
${\zap A}$  is actually to be understood as a connection $1$-form  up on 
$L^\times\subset L$ rather than as an ordinary $1$-form down on $X$. 
We can (and will) also 
restrict our choice of connection $\nabla$ by requiring that it be compatible with 
some fixed Hermitian structure $\langle ~,~\rangle$ on $L$. The curvature of such a connection has vanishing real part, 
and the space of such a connections is an affine space
modeled on 
$\mathcal{E}^1$. This in particular makes it  obvious that $\corb (L)$ belongs to real cohomology, although 
the fact that it is actually belongs to rational cohomology can of course  only be explained by invoking  other ideas.

\subsection{Self-Duality}

If $(X,g)$ is an oriented Riemannian orbifold of dimension $n$, the Hodge star operator
$\star$ may be regarded as a homomorphism 
$$\star : \Lambda^p\to \Lambda^{n-p}$$
between the $V$-bundles of differential forms of complentary degrees, and 
$$\star^2 : \Lambda^p \to \Lambda^p$$
equals  $(-1)^{p(n-p)}$ times the identity. In particular, 
if $n=4$, $\star$ is an involution of the $V$-bundle $\Lambda^2$ of $2$-forms, which 
can therefore be expressed 
as the direct sum
\begin{equation} 
\Lambda^2 = \Lambda^+ \oplus \Lambda^- 
\label{deco} 
\end{equation}
where $\Lambda^\pm$ is the $(\pm 1)$ -eigenspace of 
$\star$. Sections of $\Lambda^+$ (respectively, $\Lambda^-$) are called 
self-dual (respectively, anti-self-dual) $2$-forms. 
The decomposition (\ref{deco}) is, moreover,  {\em conformally invariant},
in the sense that it is left unchanged if the orbifold 
metric $g$ is multiplied by an arbitrary orbifold-smooth positive function. 
Any  $2$-form can thus be uniquely expressed as 
$$\varphi = \varphi^+ + \varphi^-,$$
where $\varphi^\pm \in \Lambda^\pm$, and we then have
$$ \varphi \wedge \varphi = \Big( |\varphi^+|^2 - |\varphi^-|^2\Big) d\mu_g , $$
where $d\mu_g$ denotes the metric volume form associated with our 
 orientation.

The decomposition (\ref{deco}) also leads to a decomposition of the 
 Riemann curvature tensor for any metric on an oriented $4$-dimensional orbifold. Indeed, viewing the curvature tensor
 of $g$ as a  self-adjoint linear map 
 $$
{\mathcal R} : \Lambda^2 \longrightarrow \Lambda^2
$$
we obtain a decomposition
 \begin{equation}
\label{curv}
{\mathcal R}=
\left(
\begin{array}{c|c}&\\
W_++\frac{s}{12}&\mathring{r}\\ 
&\\
\hline 
&\\
\mathring{r} & W_-+\frac{s}{12}\\
&\\
\end{array} \right)  .
\end{equation}
Here $W_+\in \End (\Lambda^+)$ is the trace-free piece of its block, and is the called
the {\em self-dual Weyl curvature} of $(M,g)$; the anti-self-dual
Weyl curvature $W_-$ is defined analogously.
Both of the objects are conformally invariant, with appropriate conformal weights. Note that 
the  {scalar curvature} $s$ is understood to act in (\ref{curv}) by scalar multiplication,
while  the  trace-free  Ricci curvature
$\mathring{r}$ 
acts on 2-forms by
contraction and projection to the alternating piece. 

Now suppose that our    oriented $4$-dimensional Riemannian orbifold
is also {\em compact}. Then  Hodge theory tells us that 
$H^2(M, \RR) \cong \mathcal{H}^2_g$, where 
$$ \mathcal{H}^2_g:=  \mathcal{H}^2(X,g)= \{ \varphi \in \mathcal{E}^2(X)~|~ d\varphi = 0, ~ d (\star \varphi ) =0\}$$
is the space of harmonic $2$-forms on $X$ with respect to the given orbifold metric $g$. However, 
$\star$ is an involution of $\mathcal{H}^2_g$, so we obtain a decomposition 
\begin{equation}
\label{codeco}
H^2(X, \RR) = \mathcal{H}^+_g \oplus \mathcal{H}^-_g
\end{equation}
where 
$$
 \mathcal{H}^\pm_g = \{ \varphi \in \mathcal{E}(X, \Lambda^\pm) ~|~ d\varphi = 0\}
$$
is the space of self-dual (respectively, anti-self-dual) harmonic $2$-forms. 
Since 
$$
\int_ X \varphi \wedge \varphi = \int_X \left(|\varphi^+ |^2 - |\varphi^- |^2\right) d\mu , 
$$
the intersection pairing 
\begin{eqnarray*}
H^2_{DR} (X, \RR) \times H^2_{DR} (X, \RR) &\longrightarrow& \RR\\
(~[\varphi ] ~,~ [\psi ]~) \quad\qquad& \longmapsto&  \int_X \varphi \wedge \psi 
\end{eqnarray*}
is therefore positive-definite on $\mathcal{H}^+$ and negative-definite on $\mathcal{H}^-$,
whereas  $\mathcal{H}^+$ and $\mathcal{H}^-$ are also orthogonal with respect to the 
intersection pairing. 
Since the intersection pairing can  be identified 
with the topologically-defined pairing
 \begin{eqnarray*}
H^2 (X, \RR) \times H^2 (X, \RR) &\longrightarrow& \RR\\
(~\mathbf{a} ~,~ \mathbf{b}~) \quad\qquad& \longmapsto&  \langle \mathbf{a}\smile \mathbf{b}, [X]\rangle 
\end{eqnarray*}
on singular cohomology, it follows that the integers
$$
b_\pm (X) = \dim \mathcal{H}^\pm
$$
are oriented  topological invariants of the Poincar\'e space $X$, and 
one  may then define the {\em signature} of $X$ to be 
$$\tau (X) = b_+(X)-b_-(X).$$
In the special case in which $X= (M, \Sigma, \beta)$, 
so that $X$ and $M$ are homeomorphic as topological spaces, 
$b_\pm (X) = b_\pm (M)$ and 
$\tau (X) = \tau (M)$ therefore coincide with  familiar topological invariants of $M$.

If $(X,g)$ is  a compact oriented $4$-dimensional Riemannian orbifold,
we can then  use the decomposition \eqref{codeco} to define a  projection
\begin{eqnarray*}
H^2 (X, \RR) &\longrightarrow&  \mathcal{H}^+_g \subset H^2 (X, \RR)\\
\mathbf{a} &\longmapsto& \mathbf{a}^+
\end{eqnarray*}
and thus associate  a non-negative 
number $(\mathbf{a}^+)^2 = \langle \mathbf{a}^+\smile \mathbf{a}^+ , [X]\rangle$ with any cohomology class 
 $\mathbf{a}\in H^2_{DR} (X, \RR)$. Of course, this number ostensibly 
still depends on $g$  via the decomposition \eqref{codeco}.

\begin{lem} \label{compare} 
Let $(X,g)$ be a compact oriented $4$-dimensional Riemannian orbifold, and 
$\mathbf{a} \in H^2(X, \RR)$. If $\psi$ is any closed orbifold $2$-form such that 
$[\psi ] = \mathbf{a}$ in de Rham cohomology, then 
$$\int_X |\psi^+|^2 d\mu \geq (\mathbf{a}^+)^2 ~$$
with equality iff $\psi$ is harmonic.
\end{lem}
\begin{proof} Let $\varphi$ be the unique harmonic $2$-form with  $[\varphi ] = \mathbf{a}$,
so that 
$$(\mathbf{a}^+)^2 = \int_X |\varphi^+|^2 d\mu.$$
Since $\varphi$ is the unique minimizer of the  $L^2$-norm in its cohomology class, 
$$\int_X (|\psi^+|^2+|\psi^-|^2) d\mu \geq \int_X (|\varphi^+|^2+|\varphi^-|^2)d\mu$$
with equality iff $\psi=\varphi$.
  On the other hand, 
$$\int_X (|\psi^+|^2-|\psi^-|^2) d\mu = \int_X (|\varphi^+|^2-|\varphi^-|^2)d\mu$$
because both sides compute the iself-ntersection number $[\psi ]^2 = \mathbf{a}^2 = [\varphi ]^2$. 
 Averaging these two formulas  then yields the desired inequality.
\end{proof}

The harmonic $2$-forms on a compact Riemannian orbifold are exactly those orbifold $2$-forms which are killed by 
the Hodge Laplacian
$$(d+d^*)^2 = - \star d \star d - d \star d \star .$$
If $\psi$ is a self-dual $2$-form, then  $\Delta_d\psi$ is also self-dual, and 
can, moreover, 
 be re-expressed by means of  the  Weitzenb\"ock formula \cite{bourg}  
  \begin{equation}
\label{friend}
\Delta_d \psi =     \nabla^{*}\nabla \psi - 2W_{+}(\psi , 
\cdot ) + \frac{s}{3} \psi  ~. 
\end{equation}
This formula leads to various interesting interplays between curvature and topology, and even 
supplies  interesting information about self-dual $2$-forms which are not assumed to 
satisfy any equation at all. Indeed, taking the 
 $L^2$ inner product of \eqref{friend} with $\psi$ tells us 
$$
  \int_{M} \left(|\nabla \psi |^{2}
-2W_{+}(\psi , \psi ) + \frac{s}{3}|\psi |^{2}\right) d\mu\geq 0 ,  
$$
since $\Delta_d$ is a  non-negative operator. 
%with equality iff $\psi$ is closed. 
On the other hand, since $W_+: \Lambda^+\to \Lambda^+$ is self-adjoint and trace-free, 
$$|W_+(\psi , \psi ) | \leq \sqrt{\frac{2}{3}}|W_+| |\psi |^2,$$
so it follows that any self-dual $2$-form $\psi$ satisfies  
\begin{equation}
\label{part} 
 \int_{M} |\nabla \psi |^{2}d\mu \geq 
\int_M\left(-2\sqrt{\frac{2}{3}}|W_+|-\frac{s}{3}\right)|\psi |^{2} d\mu .
\end{equation}
Moreover,  
 assuming that $\psi\not\equiv  0$, 
equality holds iff $\psi$ is closed, belongs the lowest eigenspace of 
$W_+$ at each point, and the two largest eigenvalues of $W_+$ are
everywhere  equal. Of course, this last assertion crucially depends on  
the fact \cite{arons,baer} that if $\Delta_d \psi =0$ and $\psi \not\equiv 0$, then $\psi\neq 0$
on a dense subset of $M$.

\subsection{Almost-Complex Structures}
\label{acs}

An {\em almost-complex structure} $J$ on an orbifold $X$ is by definition a section of the $V$-bundle $\End (TX) = T^*X \otimes TX$  such that 
$J^2 = - \mathbf{1}$. An almost-complex structure is said to be {\em integrable}, or to be a {\em complex structure}, it $X$ can be covered 
by local uniformizing charts in which it becomes the standard (constant coefficient) structure  on $\CC^m$ for some $m$. The latter happens
iff there is some  torsion-free orbifold connection $\nabla$ on $TX$ such that $\nabla J=0$. 

Suppose that $M$ is a smooth compact $4$-manifold which is equipped with an almost-complex structure $J_0$, and let 
$\Sigma\subset M$ be  a smooth compact embedded surface. Suppose  that $\Sigma$  is a pseudo-holomorphic curve with respect to $J_0$,
by which we mean that  $J_0(T\Sigma) =  T\Sigma$ at every point of $\Sigma$. Then $TX|_\Sigma$ can be made into a complex vector bundle of 
rank $2$ by equipping it with $J_0$, and this can then be split as a sum $T\Sigma \oplus N$ of complex line bundles. Now choose
any $U(1)$ connection on $N$, and use it to endow $N$ with the structure of a holomorphic line bundle by endowing it with a $\bar{\partial}$
operator. This makes $N$ into a complex manifold. Now use the tubular neighborhood theorem to endow a neighborhood $U\supset \Sigma$
with an integrable  complex
structure $J_2$ which exactly agrees with $J_0$ along $\Sigma$. In particular, the $(-i)$-eigenspace $T_2^{0,1}$ of $J_2$ is in general position
to the $(+i)$-eigenspace  $T_1^{1,0}$ of $J_0$ at $\Sigma$; hence this also holds in a neighborhood $U^\prime$  of $\Sigma$. It follows that,
on  $U^\prime$ , 
 $T_2^{0,1}$ may be expressed as
the graph of a unique section of $\phi$ of $\Lambda_0^{0,1}\otimes T^{1,0}_0= \Hom (T^{0,1}_0, T^{1,0}_0)$, where $\phi$ vanishes identically along $\Sigma$. 
Moreover, by shrinking $U^\prime$ if necessary, we may also assume that $\tr [\bar{\phi}\circ \phi] < 1/2$ everywhere. 
Now let  $f: M\to [0,1]$ be a cut-off function  which is identically $1$ on a neighborhood
of $\Sigma$ and which is  supported on a    compact subset of $U^\prime$. Then the graph of $f\phi$ is then  in general position to its conjugate, and so  is the $(-i)$-eigenspace $T_1^{0,1}$ of a unique 
 almost-complex structure $J_1$ on
$M$ which coincides with $J_0$  outside of $U^\prime$,  and coincides with $J_2$ in a neighborhood of $\Sigma$. 
Moreover, the family of tensor fields $tf\phi$, $t\in [0,1]$,  gives rise to a homotopy of almost-complex structure $J_t$ which 
interpolates between the given almost-complex structure 
 $J_0$ and the 
 constructed almost-complex structure $J_1$. All of these almost-complex structures $J_t$ exactly coincide along $\Sigma$
and outside a tubular neighborhood of $\Sigma$. 
Insofar as homotopy classes of almost-complex structures are actually the important objects for our purposes, 
the point  is  that, by merely replacing $J_0$ with a homotopic almost-complex structure $J_1$ if necessary,
 we may always assume the given almost-complex structure $J_0$ is integrable in a neighborhood of $\Sigma$. In the same way, we may also 
 assume, if necessary, that it is integrable in a neighborhood of any given finite collection of points of $M-\Sigma$.

Choose some integer $p\geq 2$, set $\beta = 1/p$, and 
consider the orbifold $(M,\Sigma, \beta)$ obtained by declaring that the total angle around 
$\Sigma$ to be $2\pi$. Of course, this is only defined modulo diffeomorphisms of $M$, so we are
free to choose the obifold charts for $(M,\Sigma, \beta)$ to be adapted to the integrable complex
structure we have just chosen on a neighborhood of $\Sigma$. Our objective now is to 
construct a homotopy class of almost-complex structures $J$ on $(M,\Sigma, \beta)$ that 
is determined by the homotopy class of almost-complex structures determined by $J_0$, 
where the homotopies on both $M$ and $(M,\Sigma, \beta)$ are both  subject to the  constraint
that $\Sigma$ is to remain a pseudo-holomorphic curve for all values of the time parameter $t$. 

We will describe two different useful ways of understanding the construction. The first of these, which 
we will call the {\em holomorphic model}, is especially useful when $J_0$ is a complex structure
on $M$, but can be carried out even when $J_0$ is merely integrable in a neighborhood of $\Sigma$. 
If  $p\in \Sigma$ is any point, let  $(w,z)\in \CC^2$ be local holomorphic coordinates on $(M,J_0)$ 
such that $z=0$ is a local defining function for 
$\Sigma$. We then introduce local uniformizing complex coordinates on $(M,\Sigma, \beta)$
by declaring that $z=\zeta^p$. This convention is often used in complex geometry \cite{cds0,rotho3}, 
because, when $(M,J_0)$ is a complex manifold,   the sheaf of holomorphic orbifold functions on  $(M,\Sigma, \beta)$ 
then coincides with the sheaf  $\mathcal{O}$ of holomorphic functions on $(M,J_0)$. We may then  equip  $(M,\Sigma, \beta)$ 
with the unique complex structure $J$ which agrees with $J_0$ on $M-\Sigma$, and coincides with the usual 
integrable complex structure tensor on $\CC^2$ in the local uniformizing charts we have just introduced. 
However,  one caveat must be borne in mind: this convention is not consistent with  
standard conventions regarding the definition of orbifolds! However, this  is actually not a serious problem.
We can  hew to the standard definition by instead using  what we'll call the  {\em oragami model}
of $(M, \Sigma , \beta )$, where we  instead 
introduce uniformizing complex coordinates $(w,\tilde{\zeta})$ such that  $z= \tilde{\zeta}^p/|\tilde{\zeta|}^{p-1}$.
Then, if we equip $M$ with a self-homeomorphism that  simply rescales the radius function $\rho$ within the 
tubular neighborhood, equals the identity outside the tubular neighborhood, is smooth on $M-\Sigma$, and 
behaves like 
$\rho \mapsto \mbox{const }\rho^{1/p}$ for small $\rho$, we then induce  a diffeomorphism between the two different models
for $(M,\Sigma, \beta)$. Notice that this self-homeomorphism of $M$ is moreover 
homotopic to the identity. 

%It is also useful to introduce a second model, which we will call the {\em topological model} of the construction. 

\begin{lem} \label{plf}
Let $M$ be equipped with an almost-complex structure $J_0$ for which $\Sigma$ is a pseudo-holomorphic curve, and 
equip $(M, \Sigma , \beta )$ with the associated homotopy class of almost-complex structures $J$. 
Then the Chern classes
 of these two spaces are related by 
$$c_1^{\rm orb}(M, \Sigma , \beta) = c_1 (M) + (\beta -1) [\Sigma ]$$
where $\beta = 1/p$. 
\end{lem}
\begin{proof}
Let $\mathcal{I}\subset \mathcal{E}_{M,\CC}$ be the 
ideal sheaf on $M$ consisting of smooth complex-valued functions on $M$ which vanish along $\Sigma$ and whose
first derivatives at $\Sigma$ are $J_0$-linear; this is a rank-$1$ free $\mathcal{E}_{M,\CC}$-module, and in fact
is exactly the sheaf of sections of the smooth line bundle $L\to M$ with $c_1(L) = -[\Sigma ]$. Similarly, let 
$\hat{\mathcal{I}}$ be the sheaf on $(X,\Sigma, \beta)$ whose sections on $V_\mathfrak{J}=U_\mathfrak{J}/\Gamma_\mathfrak{J}$ are 
 smooth complex-valued functions $f$ on $U_\mathfrak{J}$ 
which vanish along $\Sigma$,  have $J$-linear first derivatives there, and, when $\Gamma_\mathfrak{J}= \ZZ_p$,  transform under the  action of $e^{2\pi/p}\in \ZZ_p$ by 
$f\mapsto  e^{2\pi/p}f$; this is  a locally free rank-$1$ sheaf of 
$ \mathcal{E}_{(M,\Sigma, \beta),\CC}$-modules, locally generated by the complex coordinate $\zeta$,   and so is the sheaf of orbifold-smooth  sections 
of a $V$-line bundle $\hat{L}$ over $(M, \Sigma, \beta )$. Using the holomorphic model, we have a pull-back morphism $\mathcal{I}\to 
\hat{\mathcal{I}}^{\otimes p}$ induced by $(w,z) = (w, \zeta^p)$, which gives rise to an isomorphism  
$L\cong \hat{L}^{\otimes p}$. Hence 
$$\corb (\hat{L}) =  \frac{1}{p} c_1 (L) = - \frac{1}{p} [\Sigma ]$$
in rational cohomology. On the other hand, since $dz\wedge dw$ pulls back to become $p \zeta^{p-1}d\zeta\wedge dw$,
we have an induced isomorphism 
$$K = K^{orb} \otimes \hat{L}^{p-1}$$
where $K= \Lambda^{2,0}_{J_0}$ and $K^{\rm orb}= \Lambda^{2,0}_J$. Thus
\begin{eqnarray*}
c_1 (K) &=& \corb (K^{orb} ) + (p-1) \corb ( \hat{L} ) 
\\&=& \corb (K^{orb} ) + (1-\frac{1}{p}) c_1 (L ) \\
\\&=& \corb (K^{orb} ) - (1-\beta ) [\Sigma]
\end{eqnarray*}
in rational cohomology, and since 
$$
\corb (M,\Sigma, \beta) = - \corb (K^{orb} ) , \qquad  c_1 (M)  = - c_1 (K) , 
$$
we therefore have 
$$\corb (M,\Sigma, \beta) = c_1 (M)  + (\beta -1) [\Sigma ],$$
thus proving the claim.
\end{proof}

We will also need to explicitly understand the almost-complex structure $J$ in the origami model of $(M,\Sigma, \beta )$. Let us take transverse
a polar coordinate system  $({\rho}, \theta, x^1,x^2)$ about some point  of $\Sigma$. If $d\theta + \alpha$ is the imaginary part
of a $U(1)$ connection on the normal bundle $N$ of $\Sigma\subset M$, then we may take 
 $J_0$  to be integrable near $\Sigma$, and given by 
 \begin{eqnarray*}
J_0 =  \frac{\partial}{\partial \theta} \otimes  \frac{d\rho}{\rho}-  \rho\frac{\partial}{\partial \rho} \otimes d\theta &+& 
\left(
\frac{\partial}{\partial x^2} -\alpha_2\frac{\partial}{\partial \theta} -{\alpha_1}{\rho}\frac{\partial}{\partial \rho} 
\right) \otimes dx^1\\ &-&\left(
\frac{\partial}{\partial x^1} -\alpha_1\frac{\partial}{\partial \theta} +{\alpha_2}{\rho}\frac{\partial}{\partial \rho} 
\right) \otimes dx^2
\end{eqnarray*}
in our transverse polar coordinates. Passing to orbifold coordinates near $\Sigma$ just involves replacing the 
polar angle $\theta$ with a new polar angle $\tilde{\theta}= \theta/p$, so that the pull-back of the above becomes 
 \begin{eqnarray*}
J_0 =  \frac{\partial}{\partial \tilde{\theta}} \otimes  p \frac{d\rho}{\rho} -   \frac{1}{p}\rho\frac{\partial}{\partial \rho} \otimes d\tilde{\theta} &+& 
\left(
\frac{\partial}{\partial x^2} -p\alpha_2\frac{\partial}{\partial \tilde{\theta}} -{\alpha_1}{\rho}\frac{\partial}{\partial \rho} 
\right) \otimes dx^1\\ &-&\left(
\frac{\partial}{\partial x^1} -p\alpha_1\frac{\partial}{\partial \tilde{\theta}} +{\alpha_2}{\rho}\frac{\partial}{\partial \rho} 
\right) \otimes dx^2
\end{eqnarray*}
By contrast, the branched-cover complex structure is given by 
 \begin{eqnarray*}
\frac{\partial}{\partial \tilde{\theta}} \otimes  \frac{d\rho}{\rho}-  \rho\frac{\partial}{\partial \rho} \otimes d\tilde{\theta} &+& 
\left(
\frac{\partial}{\partial x^2} -p\alpha_2\frac{\partial}{\partial \tilde{\theta}} -p{\alpha_1}{\rho}\frac{\partial}{\partial \rho} 
\right) \otimes dx^1\\ &-&\left(
\frac{\partial}{\partial x^1} -p\alpha_1\frac{\partial}{\partial \tilde{\theta}} +p{\alpha_2}{\rho}\frac{\partial}{\partial \rho} 
\right) \otimes dx^2
\end{eqnarray*}
and one can interpolate between these two by taking our orbifold complex structure to be the integrable complex structure
\begin{eqnarray*}
J =  \frac{\partial}{\partial \tilde{\theta}} \otimes  f(\rho ) \frac{d\rho}{\rho} -   \frac{1}{f(\rho) }\rho\frac{\partial}{\partial \rho} \otimes d\tilde{\theta} &+& 
\left(
\frac{\partial}{\partial x^2} -p\alpha_2\frac{\partial}{\partial \tilde{\theta}} -\frac{p}{f(\rho) }{\alpha_1}{\rho}\frac{\partial}{\partial \rho} 
\right) \otimes dx^1\\ &-&\left(
\frac{\partial}{\partial x^1} -p\alpha_1\frac{\partial}{\partial \tilde{\theta}} +\frac{p}{f(\rho) }{\alpha_2}{\rho}\frac{\partial}{\partial \rho} 
\right) \otimes dx^2
\end{eqnarray*}
where $f(\rho)$ is a smooth  positive function with $f\equiv 1$ for, say,  $\rho < \epsilon$ and $f\equiv p$ for, say, $\rho > 10 \epsilon$.  
In fact, this $J$ is simply the pull-back of $J_0$ via a suitable homeomorphism which is smooth away from $\Sigma$; indeed, if we 
set ${\zap r}= \exp \frac{1}{p} \int f(\rho) d\rho /\rho$, then $J$ becomes $J_0$, with $\rho$ replaced by ${\zap r}$. Since, for an appropriate 
choice of constant of integration,  
${\zap r} = \mbox{const }\rho^{1/p}$ for small $\rho$ and ${\zap r} = \rho$ for large $\rho$, we see that this harmonizes 
the holomorphic and oragami pictures in exactly the manner previously promised.

\subsection{Symplectic Structures}

\begin{prop} Let $(M, \omega_0)$ be a symplectic manifold, and suppose that  $\Sigma \subset M$ 
is an embedded surface to which $\omega_0$ restricts as an area form. Choose any  integer $p\geq 2$, 
and set $\beta =1/p$. Then the orbifold $(M,\Sigma, \beta)$ also admits a symplectic form $\omega$
admits a symplectic form $\omega$ with $[\omega]=[\omega_0]$ in $H^2(M,\RR)$. 
\end{prop}
\begin{proof}
By the Weinstein tubular neighborhood theorem \cite{weinsteinlec}, 
a tubular neighborhood of $\Sigma$ is determined up to symplectomorphism by 
the induced area form and the symplectic normal bundle of $\Sigma$. 
Thus, choosing a complex line bundle $E\to \Sigma$ of degree $[\Sigma]^2$, 
the symplectic form is expressible in transverse polar coordinates  $({\rho}, \theta, x^1,x^2)$ about some point  of $\Sigma$ as
\begin{eqnarray*}
\omega &= &d\left( \frac{{\rho}^2}{2}  (d\theta + \alpha) \right)+ \varpi^* \omega_\Sigma \\&=&   d{\rho}\wedge {\rho} (d\theta + \alpha) 
+ \frac{1}{2} {\rho}^2\varpi^*\Omega   + \varpi^* \omega_\Sigma
\end{eqnarray*}
where $\omega_\Sigma$ is the induced area form on $\Sigma$, $\varpi : E\to \Sigma$
is the bundle projection, and 
$d\theta + \alpha$ is the imaginary part of a  $U(1)$ connection form on $E$, with curvature $i\Omega= id\alpha$, 
expressed for concreteness in terms of a local connection form $\alpha$ on $\Sigma$. 
If we now set $\tilde{\theta} = \beta  \theta$ and  ${\zap r} = {\rho}/\sqrt{\beta}$, where $\beta = 1/p$, we then have 
\begin{eqnarray*}
\omega &= &d\left( \frac{{\zap r}^2}{2}  (d\tilde{\theta} + p\alpha) \right)+ \varpi^* \omega_\Sigma \\&=&   d{\zap r} \wedge {\zap r} (d\tilde{\theta} +p \alpha) 
+ \frac{1}{2} {\zap r}^2\varpi^*(p\Omega )   + \varpi^* \omega_\Sigma
\end{eqnarray*}
which may be viewed, in the origami model,  as an orbifold symplectic structure on the tubular neighborhood.
\end{proof}

The key point  is that there is no {\em symplectic} difference between a $2$-dimensional cone and a $2$-dimensional disk. The fact that
they are {\em metrically} different reflects different choices of almost-complex structure. Now notice that $\omega$ is 
invariant under the action of the almost-complex structure $J$ explicitly described in the oragami model at the end 
of \S \ref{acs}, and that $\omega (\_ , J\_ )$ is moreover positive-definite. In the symplectic case, this gives a self-contained 
characterization of the homotopy class of $J$ on $(M, \Sigma, \beta )$.

\subsection{The Todd Genus}

If  $(M,J_0)$ is a complex surface, and if $\Sigma \subset M$ is a holomorphic curve, then, for any integer 
$p\geq 2$,
the so-called holomorphic model of $(M, \Sigma, \beta)$, $\beta = 1/p$,  
has some remarkable advantages. In particular, the structure sheaf 
${\mathcal O}_{(M, \Sigma , p)}$ is actually just equal to the usual structure sheaf ${\mathcal O}_M$
of holomorphic functions on the complex manifold $M$. Indeed, a holomorphic
function $f (w, \zeta)$ is invariant under the action of $(w, \zeta)  \mapsto (w , e^{2\pi i /p}\zeta )$
iff it can be expressed as $f(w, z)$, where $z= \zeta^p$. The interesting point is 
that the orbifold point of view then leads to a non-standard acyclic resolution 
$$
0\to \mathcal{O} \to \mathcal{E}^{0,0}_{(M, \Sigma, \beta )} \stackrel{\bar{\partial}}{\rightarrow} \mathcal{E}^{0,1}_{(M, \Sigma, \beta )} \stackrel{\bar{\partial}}{\rightarrow} \mathcal{E}^{0,2}_{(M, \Sigma, \beta )} \to 0
$$
of the structure sheaf, and so the abstract de Rham theorem tells us that 
$$
H^{0,k} (M, \Sigma, \beta) = H^k (M, \mathcal{O}) = H^{0,k}(M)
$$
for every integer $k$. In particular, the index of the  orbifold elliptic operator 
$$
\bar{\partial} + \bar{\partial}^*: \mathcal{E} ((M, \Sigma, \beta ), \Lambda^{0,0} \oplus \Lambda^{0,2}) \to \mathcal{E} ((M, \Sigma, \beta ), \Lambda^{0,1}) 
$$
is exactly the Todd genus of the original manifold $M$: 
$$\chi ((M, \Sigma, p), {\mathcal O}) = \chi (M, {\mathcal O}) = \frac{(\chi + \tau )(M)}{4} .$$

We can put this in a broader context by considering considering thinking of the $V$-bundles
\begin{eqnarray*}
\mathbb{V}_+&=& \Lambda^{0,0} \oplus \Lambda^{0,2}\\
\mathbb{V}_-&=& \Lambda^{0,1} 
\end{eqnarray*}
as actually being twisted versions 
$$
\mathbb{V}_\pm = \mathbb{S}_\pm \otimes L^{1/2}
$$
of the spin bundles for some orbifold metric $g$ adapted to $J$;
here $L=K^{-1}$ is the anti-canonical $V$-bundle. If  $g$  K\"ahler, then $\sqrt{2}(\bar{\partial} + \bar{\partial}^*)$
is then just \cite{hitharm} the spin$^c$ Dirac operator  ${\slashed{D}}$ associated to an appropriate connection on $L$.
However, even for more general metrics, the ${\slashed{D}}$ and $\sqrt{2}(\bar{\partial} + \bar{\partial}^*)$ will have
the same symbol, and hence the same index. This idea naturally leads to the following key observation:

\begin{prop} \label{toddgenus}
Let $(M, J_0)$ be a $4$-manifold with  almost-complex structure, and let $\Sigma\subset M$ be any compact embedded pseudo-holomorphic
curve. Let $p\geq 2$ be an integer, let $\beta = 1/p$, and let $(M, \Sigma , \beta)$ be the orbifold obtained from 
$M$ by declaring the total angle around $\Sigma$ to be $2\pi \beta$. Let $J$ be an orbifold almost-complex structure
on $(M, \Sigma , \beta)$ in the homotopy class discussed in \S \ref{acs}.  Let $\slashed{D}_0$ be a spin$^c$ Dirac operator on $M$ for the spin$^c$ structure
induced by $J_0$, and let  ${\slashed{D}}$ be  a spin$^c$ Dirac operator on $(M,\Sigma, \beta)$ for the  spin$^c$ structure
induced by ${J}$. Then ${\slashed{D}}$ and $\slashed{D}_0$  have the same index: 
$$
\Ind ({\slashed{D}}) = \Ind (\slashed{D}_0) = \frac{(\chi + \tau)(M)}{4}~.
$$
\end{prop}
\begin{proof} The index theorem for elliptic operators on orbifolds \cite{kawasaki} implies that the difference $\Ind ({\slashed{D}}) - \Ind (\slashed{D}_0)$ is expressible
in terms of the Euler characteristic and self-intersection of $\Sigma$, since these numbers also determine the restriction of the symbol to the singular set.  
However, we have already seen that $\Ind ({\slashed{D}}) - \Ind (\slashed{D}_0)=0$ if $\Sigma\subset M$ is a holomorphic curve in a compact
complex surface. Since $\chi (\Sigma )$ and $[\Sigma ]^2$ take all possible values 
 in such examples, it follows that $\Ind ({\slashed{D}}) - \Ind (\slashed{D}_0)$
must vanish in full generality. 
\end{proof}

\subsection{The Generalized Hitchin-Thorpe Inequality}

The Euler characteristic $\chi$ and signature $\tau$ of a smooth compact $4$-manifold $M$ 
 may both be calculated by choosing any smooth Riemannian metric $g$ on $M$,
 and then integrating appropriate universal quadratic polynomials
 in the curvature of $g$. When $g$ has an edge-cone singularity, however, 
 correction terms must be introduced in order to compensate for the singularity 
 of the metric along the given surface $\Sigma\subset M$.  In 
  \cite{atleb}, the following formulas were proved for any edge-cone metric $g$ of cone angle $2\pi \beta$ on  a pair $(M,\Sigma)$,
  where $M$ is a smooth compact oriented  $4$-manifold, and $\Sigma \subset M$ is a smoothly embedded  compact oriented surface: 
 
\begin{eqnarray}  \label{gabo}
 \chi (M)  - (1-\beta) \chi (\Sigma ) &=& 
 \frac{1}{8\pi^2}\int_M \left(
 \frac{s^2}{24} + |W_+|^2 + |W_-|^2 -\frac{|\mathring{r}|^2}{2}
 \right) d\mu \qquad \\
 \label{thom}
 \tau (M)  - \frac{1}{3} (1-\beta^2) [\Sigma ]^2 &=& 
 \frac{1}{12\pi^2}\int_M \left(
  |W_+|^2 - |W_-|^2
 \right) d\mu ~.
\end{eqnarray}
In particular, it follows that these formulas are valid when $\beta = 1/p$ and $g$ is an orbifold metric on $(M,\Sigma, \beta )$. 
Indeed, the orbifold versions of these  formulas  are implicit in the earlier work of other authors \cite{izawa,kawasaki0,satake}, and one
of the two proofs of (\ref{gabo}-\ref{thom}) given in \cite{atleb} shows that  the validity of these formulas in general  is actually a logical consequence of their 
valdity for orbifolds. 

The Hitchin-Thorpe inequality \cite{bes,hit,tho} provides an important obstruction to the 
existence of Einstein metrics on $4$-manifolds, and this fact has a natural generalization \cite{atleb} to the 
setting of edge-cone metrics. Indeed, if $g$ is an edge-cone metric of cone angle $2\pi \beta$ on $(M,\Sigma)$,
equations  (\ref{gabo}-\ref{thom}) tell us that 
\begin{eqnarray*}
2 \Big[ \chi (M) - (1- \beta)\chi (\Sigma)
\Big] &\pm&  3 \Big[ \tau (M) - \frac{1}{3}(1- \beta^2) [\Sigma ]^2\Big]\\
&=&
\frac{1}{4\pi^2}\int_M \left(
 \frac{s^2}{24} + 2|W_\pm|^2 -\frac{|\mathring{r}|^2}{2} \right) d\mu 
\end{eqnarray*}
so that the topological expression on the left-hand-side is necessarily non-negative if the given metric $g$ is Einstein.

In this article, we are primarily interested in the special case of $\beta=1/p$, where  
Proposition \ref{reggie}  tells us that any Einstein edge-cone metric is actually an orbifold metric. 
Our focus will be on the case when $M$ admits an almost-complex structure $J_0$
for which $\Sigma$ is a pseudo-holomorphic curve;  as we  saw in \S \ref{acs}, 
the orbifold $(M, \Sigma, \beta)$ then carries an almost-complex structure $J$ which 
is induced by $J_0$. Now notice that the orbifold Chern class of this line bundle satisfies 
\begin{eqnarray*}
(\corb )^2 (M, \Sigma , \beta) &=& (c_1+ (\beta -1) [\Sigma ])^2 
\\&=& c_1^2 + 2(\beta -1) c_1\cdot [\Sigma ] + (\beta -1)^2 [\Sigma ]^2
\\&=& c_1^2 + 2 (\beta -1)\chi (\Sigma) + 2 (\beta -1)  [\Sigma ]^2   + (1-2 \beta +  \beta^2) [\Sigma ]^2
\\&=& c_1^2 - 2 (1- \beta)\chi (\Sigma)- (1-\beta^2) [\Sigma ]^2
\\&=& 2 \Big[ \chi (M) - (1- \beta)\chi (\Sigma)
\Big] + 3 \Big[ \tau (M) - \frac{1}{3}(1- \beta^2) [\Sigma ]^2
\Big] 
\end{eqnarray*}
so that our previous computation tells us that  this quantity  can be expressed in terms of  the curvature of an arbitrary metric $g$
$(M, \Sigma, \beta)$
by 
\begin{equation}
\label{urbi}
(\corb )^2 (M, \Sigma , \beta) =  \frac{1}{4\pi^2}\int_M \left(
 \frac{s^2}{24} + 2|W_+|^2 -\frac{|\mathring{r}|^2}{2}
 \right) d\mu 
\end{equation}
and that $(\corb)^2$ is therefore necessarily non-negative if  $(M, \Sigma , \beta) $ admits an Einstein case. This gives us  
a useful  special case of \cite[Theorem A]{atleb}:

\begin{prop}\label{htal} Let $(M,J_0)$ be a smooth compact $4$-manifold with almost-complex structure, let 
$\Sigma \subset M$ be a compact embedded pseudo-holomorphic curve. Choose an integer $p\geq 2$, and  equip
the orbifold $(M, \Sigma, \beta)$, $\beta = 1/p$, with  the almost-complex structure $J$ which is determined, up to homotopy, 
by $J_0$. If $(M, \Sigma, \beta)$ admits an orbifold Einstein metric $g$, then 
$(\corb )^2(M, \Sigma, \beta)\geq 0$, with equality iff $g$ is Ricci-flat and anti-self-dual. 
\end{prop}

Requiring   that $g$  be Ricci-flat ($r\equiv 0$) and anti-self-dual ($W_+\equiv 0$) is equivalent to asking that $\Lambda^+$ be flat.
This happens precisely when $g$ has restricted holonomy $\subset SU(2)$, and amounts to saying that 
$g$ is {\em locally hyper-K\"ahler}.  The next two results provide all the global information about such  orbifolds
that  will be  needed to prove Theorems \ref{non} and \ref{oui}.

\begin{prop}\label{exclude}
Let $M$ be a smooth compact oriented $4$-manifold such that $b_+(M) \neq 0$, and let $\Sigma \subset M$ be a non-empty  smooth 
compact oriented  surface, possibly with several connected components. For some integer $p\geq 2$, set $\beta = 1/p$,  and let $(M, \Sigma, \beta)$ be 
the orbifold obtained by declaring that the total angle around $\Sigma$ is $2\pi \beta$. Suppose that $(M, \Sigma, \beta)$ does not
admit orbifold metrics of positive scalar curvature, but does admit  a scalar-flat anti-self-dual  orbifold metric $g$. 
Then $b_+(M)=1$, $p\in \{ 2, 3, 4\}$, and $g$ is  K\"ahler and  Ricci-flat. Moreover,  $M$ carries an integrable  complex structure $J_0$ 
such that $\Sigma\subset M$ is  a holomorphic curve, and $(M,J_0)$ is either 
 a rational complex surface
or a  finite quotient of  $\CP_1$ times  an  elliptic  curve.  
Finally, the orbifold $(M, \Sigma, \beta)$ is a global quotient: if $g$ is  flat, it is the product of two elliptic curves, divided by  a finite  group; otherwise, 
it is a Calabi-Yau  $K3$  divided by  an isometric action of $\ZZ_p$.
\end{prop}
\begin{proof}
Since $b_+(M)= b_+(M, \Sigma, \beta)$ is non-zero, the orbifold $(M, \Sigma, \beta)$ carries some non-trivial self-dual harmonic $2$-form 
$\psi$. Because we have assumed that $s$ and $W_+$ vanish identically, the Weitzenb\"ock formula \eqref{friend} simplifies to read
$$0= (d+d^*)^2 \psi =  \nabla^* \nabla \psi$$
so that $\int_M |\nabla \psi|^2 d\mu$  vanishes, and every such harmonic form $\psi$ must therefore 
be parallel. However,  any point $p\in \Sigma\subset M$ is covered by a
local uniformizing chart  with local uniformizing goup $\ZZ_p$; moreover, we can take this chart to consist of  geodesic normal coordinates
about $p$, so that $\ZZ_p< SO(2)$ acts on $\RR^4$ by rotation about $\RR^2$. But   $SO(2) < SO(4)$, and hence $\ZZ_p< SO(4)$, acts 
faithfully on $\Lambda^+$ via rotation about an axis;
in particular, this action preserves only a $1$-dimensional subspace of $\Lambda^+$. It follows that there is at most  a $1$-dimensional 
space of parallel  self-dual $2$-forms $\psi$ on $(M,\Sigma, \beta)$, and since this space coincides with $\mathcal{H}^+_g\neq 0$,
it follows it is exactly $1$-dimensional. In particular, $b_+(M)=1$. Moreover, since the point-wise norm of a parallel form $\psi$ is constant, 
there is, up to sign, a unique parallel self-dual $2$-form $\omega$ on $M$ with $|\omega | \equiv \sqrt{2}$. 
This form is the K\"ahler form associated with a unique complex structure $J$ on $(M,\Sigma, \beta)$, and in a locally uniformizing chart 
near any $p\in \Sigma$, this become a complex structure on $\RR^4$ which is preserved the local $\ZZ_p$ action. In particular, 
 the local $\ZZ_p$-action   is by holomorphic maps with fixed point set $\Sigma$, so that $\Sigma$ is  a holomorphic curve in  local 
 complex coordinates, and $M$ can locally be thought of as  a $p$-fold cyclic branched quotient. Choosing two generators for
 the  $\ZZ_p$-invariant  local holomorphic 
 functions now gives us complex coordinates on $M$. This equips $M$ with an integrable almost-complex structure $J_0$,  and turns 
 $\Sigma\subset M$ into  a complex curve
 in the compact complex surface $(M,J_0)$. Moreover, since $b_+(M)=1$ is odd, $(M,J_0)$ is necessarily   \cite{bpv} of K\"ahler type. 
 
 Since we have also assumed that $(M,\Sigma, \beta)$ does not admit any metrics of positive scalar curvature, an argument due to 
 Bourguignon  \cite{bes,bouric} shows that our scalar-flat metric must be Ricci-flat; indeed, one could otherwise produce a
 metric with $s> 0$ by following the Ricci flow for a short time, and then conformally rescaling by the lowest eigenfunction of the Yamabe Laplacian. 
 Thus $(M, \Sigma, \beta ), g, J)$ is actually a Ricci-flat K\"ahler orbifold, and in particular has $\corb =0$. 
 Lemma \ref{plf}  therefore now  tells us that 
 $$ 0 = \corb  (M,\Sigma, \beta) = c_1(M,J_0) + \left(\frac{1}{p} -1\right) [\Sigma ] $$
 in $H^2(M, \QQ)$. 
 In particular,  the canonical line bundle $K$ of the compact complex surface $(M,J_0)$ has negative degree respect to 
 any K\"ahler form $\omega_0$, so no positive power of $K$ can have a holomorphic section. Thus, the Kodaira dimension of the K\"ahler surface
 $(M,J)$ is $-\infty$, and surface classification \cite{bpv,gh} tells us that $(M, J_0)$ is either rational or ruled. 
 
 In particular, $H_1(M,\ZZ)$ is
 torsion-free, and the same therefore goes for $H^2(M,\ZZ )$. The equation 
 $$
  [\Sigma ] =  p \Big( [\Sigma ] - c_1 (M, J_0)\Big) 
 $$
 is therefore valid in integer cohomology, rather than just rationally. It follows that the divisor line bundle $D$ of
 $\Sigma$ has a $p^{\rm th}$ root $E\to M$, so that  $D= E^{\otimes p}$ as holomorphic line bundles. 
  We can therefore construct  a $p$-fold cyclic branched covering $\varpi: Y\to M$
 branched along $\Sigma$ by setting
  $$Y:= \{ \zeta \in E~|~ \zeta^{\otimes p} = f\}
 $$ 
 where $f\in H^0(M, \mathcal{O} (D))$ vanishes exactly at $\Sigma$.
The pull-back $\hat{g}= \varpi^* g$  then makes $(Y,\hat{g})$ into a compact Ricci-flat K\"ahler surface, 
and displays $\left((M, \Sigma, \beta ), g, J\right)$ as a  global quotient $(Y,\hat{g})/\ZZ_p$ of some locally hyper-K\"ahler manifold. 
In particular, surface classification tells us that $Y$ is finitely covered by either $K3$ or $T^4$. Thus  $\chi (Y) \leq 24$, with 
equality iff $Y$ is a $K3$ surface. 

The possible values of $p$ are severely constrained. Indeed, since
$$
 c_1 (M, J_0)  =  (p -1) \Big([\Sigma ] - c_1 (M, J_0) \Big)
 $$
 in integer cohomology, $c_1 (M, J_0)$ must be divisible by $(p-1)$. If $(M,J_0)$ is ruled, and if ${\zap F}$ is the fiber class, 
 we have  $c_1\cdot {\zap F}=2$, so $p-1$ divides $2$, and $p$ is either $2$ or $3$. If $(M,J_0)$ is non-minimal, 
 it contains an exceptional curve ${\zap E}$ on which $c_1\cdot {\zap E}=1$, so $p-1$ divides $1$, and $p=2$. Finally, if
 $M=\CP_2$, it contains a line ${\zap L}$ on which $c_1\cdot {\zap L}=3$, so $p-1$ divides $3$, and $p=2$ or $p=4$. 
 
 Now consider what happens if $b_1(M) =0$. In this case,  $h^{1}(M, \mathcal{O})=h^{2}(M, \mathcal{O})=0$, so 
  $\chi (M, \mathcal{O})=1$. Now,  since $Y$ is a $p$-fold branched cyclic cover, 
  $$\chi (Y)= p \chi (M) + (1-p) \chi (\Sigma).$$  On the other hand, $\corb = c_1(M) + (\beta -1) \Sigma =0$, so 
  $\chi (\Sigma ) + \beta [\Sigma ]^2=0$  by adjunction, and hence $\chi(\Sigma) = -[p/(p-1)^2]c_1^2$. We therefore conclude that
 \begin{eqnarray*}
\chi (Y)
&=& \frac{p}{p-1}\Big[ (p-1)  \chi (M) +c_1^2(M)\Big] \\
&=& \frac{p}{p-1}\Big[ (p-2)  \chi (M) +12 \chi (M, \mathcal{O})\Big] \\
&=& \frac{p}{p-1}\Big[ (p-2)  \chi (M) +12 \Big] .
\end{eqnarray*}
If $p=2$, it follows that $\chi(Y)=24$. 
If $p=3$, $M$ must be a Hirzebruch surface, with $\chi (M)=4$, so again $\chi(Y)=24$.
If $p=4$, $M$ must be $\CP_2$, with $\chi (M)=3$, and once again $\chi (Y)=24$. Thus, the hyper-K\"ahler manifold 
$Y$ must be a $K3$ surface whenever $b_1(M) =0$.

Now suppose that $b_1(M)\neq 0$. Then there is a holomorphic $1$-form $\varphi \not\equiv 0$ on $(M, J)$, and 
this pulls back to a a holomorphic $1$-form $\hat{\varphi}$ on $Y$. However, $Y$ is Ricci-flat, so $\hat{\varphi}\not\equiv 0$ is 
both parallel  and   $\ZZ_p$-invariant. It follows that $(Y, \hat{g})$ is flat, and that 
$\varphi$ restricts to a $\Sigma$ as a non-zero holomorphic $1$-form. Each component of $\Sigma$ is therefore
a $2$-torus, and submerses holomorphically onto the base of our ruled surface via the Albanese map. Thus $(M,J_0)$ must be a ruled surface over 
an elliptic curve $T^2$. Moreover, since $\varphi$ is non-zero everywhere, every fiber of $M\to T^2$ is non-singular; and since 
$\varphi$ is 
 parallel, every fiber of the 
the ruling is totally geodesic, and so is itself a flat orbifold. This means that the fibers of $Y\to T^2$ are elliptic curves which are $p$-to-$1$ cyclic branched covers
of the $\CP_1$ fibers of $M\to T^2$, with $4$ branched points if $p=2$, or $3$ branched points if $p=3$. 
The parallel  $(1,0)$-vector field $\xi=\bar{\varphi}^\sharp$ on $Y$ is necessarily holomorphic and $\ZZ_p$-invariant, so $Y$ also carries 
a $\ZZ_p$-invariant holomorphic foliation transverse to the fibers of $Y\to T^2$, and this foliation induces a flat projective connection on 
$M\to T^2$ with monodromy consisting of permutations of the  branch points in which $\Sigma$ meets the fibers. Pulling this back
to a finite cover of $T^2$ then gives us a finite cover of $M$ biholomorphic to $\CP_1\times T^2$, and the corresponding cover of $Y$
is then biholomorphic to the product of two elliptic curves. 
\end{proof}

\begin{prop}
Let $M$ be a smooth compact oriented $4$-manifold such that  $b_+(M) \neq 0$, and let $\Sigma \subset M$ be a (non-empty)  smooth 
compact oriented  surface. For some integer $p\geq 2$, set $\beta = 1/p$,  let $(M, \Sigma, \beta)$ be 
the orbifold obtained by declaring that the total angle around $\Sigma$ is $2\pi \beta$, and suppose that 
$(M, \Sigma, \beta)$ admits a scalar-flat anti-self-dual  orbifold metric $g$. Then $g$ is K\"ahler, and $b_+(M)=1$.
Moreover, $g$ is Ricci-flat iff $(M, \Sigma, \beta)$ does
not admit orbifold metrics of positive scalar curvature. 
\end{prop}
\begin{proof} Given that $g$ is anti-self-dual and scalar-flat, the Weitzenb\"ock formula \eqref{friend} again shows that every self-dual 
harmonic $2$-form is parallel. Since $b_+(M) \neq 0$,  there must be at least one non-zero parallel self-dual form, so $g$ is globally K\"ahler. 
If it is not Ricci-flat, then $\Lambda^+$ is not flat, and so there cannot be a self-dual harmonic form which is linearly independent from the
first one, so $b_+(M)=1$; moreover, Bourguignon's argument \cite{bes,bouric} shows that $(M,\Sigma, \beta)$ also admits positive scalar curvature metrics. 

 On the other hand, if $g$ is Ricci-flat K\"ahler, the proof of Proposition \ref{exclude} shows that we still have $b_+(M)=1$, but that 
$(M, \Sigma, \beta)$  is in this case a global quotient of $K3$ or $T^4$ by a finite group. However, neither $K3$ nor $T^4$ admits
Riemannian metrics of positive scalar curvature \cite{bes, gvln2, spccs, lic}. Since orbifold metrics on   $(M, \Sigma, \beta)$  are really just
Riemannian metrics on $K3$ or $T^4$ which are invariant under the action of the appropriate finite group, it follows that $(M, \Sigma, \beta)$
then does not admit orbifold metrics of positive scalar curvature. 
\end{proof}

\subsection{Almost-K\"ahler Geometry}

A rather special set of techniques can be applied if, for a given  the orbifold metric $g$, there  happens to 
be a harmonic  self-dual 2-form $\omega\in {\mathcal H}^+_g$ with constant 
point-wise norm  $|\omega|_g \equiv \sqrt{2}$. In this case, there is 
an associated orbifold almost-complex structure $J:TM\to TM$, $J^2 =1$, defined by 
$$ g (J\cdot, \cdot ) = \omega (\cdot , \cdot ) ,$$
and this almost-complex structure then acts on $TM$ in a $g$-preserving fashion. 
The triple $(M,g,\omega )$ is then said to be an {\em almost-K\"ahler $4$-orbifold}. 
Because $J$ allows one to to think of $TM$ as a complex vector bundle, 
it is only natural to look for a connection on its anti-canonical $V$-line bundle 
$L=\wedge^2 T^{1,0}_J\cong \Lambda^{0,2}_J$ in order to use 
the Chern-Weil theorem in order to  express $\corb (J)$ as 
$$\corb (J)= [\frac{i}{2\pi}F]\in H^2_{DR}(M, \RR)~,$$ 
where $F$ is the curvature of
the relevant connection on $L$. A particular choice of 
Hermitian connection on $L$ was  first introduced in the manifold context by Blair \cite{blair}, and   later rediscovered
by Taubes \cite{taubes} for entirely different reasons. Of course, all the local calculations this entails 
are also valid for orbifolds. In partocular, the curvature
$F_{\zap B}=F_{\zap B}^++F_{\zap B}^-$ of 
this {\em Blair connection}  is given  \cite{td,lsymp}  by 
\begin{eqnarray}
iF_{\zap B}^+ &=& \frac{s+s^*}{8}\omega+  W^+(\omega )^\perp   \label{bla1}\\
iF_{\zap B}^-&=&  \frac{s-s^* }{8}\hat{\omega}+ \mathring{\varrho} \label{bla2}
\end{eqnarray}
where  the so-called {\em star-scalar curvature} is given by 
$$s^*=s+ |\nabla \omega |^2 = 2W_+ (\omega , \omega ) +  \frac{s}{3}~,$$
while   $W^+(\omega )^\perp$
  is the component of $W^+(\omega )$ orthogonal to $\omega$,  
  $$\mathring{\varrho}(\cdot , J\cdot )  = \frac{\mathring{r} +  J^*\mathring{r} }{2} ,$$
and where  the  anti-self-dual $2$-form 
 $\hat{\omega}\in \Lambda^-$
 is  defined only on the open set where $s^* - s\neq 0$,  and   satisfies 
 $|\hat{\omega}|\equiv \sqrt{2}.$

An important special case occurs when $\nabla \omega=0$. This happens precisely
when 
$J$ is integrable, and $g$ is a  K\"ahler metric compatible with $J$. In this case, $s=s^*$, $\omega$ is an eigenvector of the 
$W_+$,  $r$ is $J$-invariant, and $iF_{\zap B}$ is the 
Ricci form of $(g,J)$. In fact,
$\omega$ is an eigenvector of $W_+$ with eigenvalue $s/6$, whereas the 
elements of $\omega^\perp= \Re e \Lambda^{2,0}_J$ are eigenvectors of eigenvalue $-s/12$. 

K\"ahler metrics of  constant negative scalar curvature $s$  play 
a privileged  role in $4$-dimensional geometry. For our purposes, though, it will sometimes
be useful to regard them as belonging to the following  broader class of almost-K\"ahler
metrics:

\begin{defn} \label{fatso} 
An almost-K\"ahler metric $g$  on a $4$-dimension orbifold    will be said to be 
{\em saturated} if 
 \begin{itemize}
  \item $s+s^*$ is a negative constant; 
 \item the associated symplectic form $\omega$ belongs to the lowest eigenspace of $W_+: \Lambda^+\to \Lambda^+$ 
at each point; and 
 \item the two largest eigenvalues of  $W_+: \Lambda^+\to \Lambda^+$ 
 are   everywhere equal.
 \end{itemize}
\end{defn}

\section{Seiberg-Witten Theory}

Suppose that $M$ is a smooth compact oriented $4$-manifold, let $J_0$ be an 
almost-complex structure on $M$, and let $\Sigma\subset M$ be a pseudo-holomorphic curve;
that is, let $\Sigma$ be  a compact, smoothly embedded surface such that $J_0(T\Sigma) = T\Sigma$. We now
choose some integer $p\geq 2$, set $\beta = 1/p$, and let $(M,\Sigma, \beta)$ be the smooth compact
oriented orbifold with underlying topological space $M$, regular set $M-\Sigma$, and with total angle
$2\pi \beta$  around $\Sigma$. 
As we have
saw in \S \ref{acs}, we can then endow $(M,\Sigma, \beta)$ with an almost-complex structure
$J$ which agrees with $J_0$ outside a tubular neighborhood of $\Sigma$, and such that
$\Sigma$ is also a pseudo-holomorphic curve with respect to $J$. 
This allows us to 
define vector $V$-bundles $\Lambda^{0,k}$, and we can then set 
\begin{eqnarray*}
\mathbb{V}_+&=& \Lambda^{0,0} \oplus \Lambda^{0,2}\\
\mathbb{V}_-&=& \Lambda^{0,1}
\end{eqnarray*}
so that $\det \mathbb{V}_+= \det \mathbb{V}_-=  \Lambda^{0,2}$.
If we choose an orbifold metric $g$ which is $J$-invariant, these bundles
then become the twisted spinor bundles of an orbifold spin$^c$ structure
on $(M,\Sigma, \beta)$, and we formally have 
$$
\mathbb{V}_\pm =  \mathbb{S}_\pm  \otimes L^{1/2} 
$$
where $L=K^{-1}\cong \Lambda^{0,2}$ is the anti-canonical line bundle of $J$. 
Notice that, by Lemma \ref{plf}, we have 
$$
\corb (L) = c_1 (M, J_0) + (\beta -1) [\Sigma ].$$
Since any other orbifold metric on our orbifold is obtained from $g$ by a self-adjoint automorphism
of the tangent bundle, this construction also induces unique choices of  twisted spinor bundles $\mathbb{V}_\pm$ for any other metric
on $(M,\Sigma, \beta)$. 

We now endow $L$ with a fixed Hermitian inner product $\langle~,~\rangle$. 
Every 
unitary 
connection ${\zap A}$ on then $L$ induces a unitary connection 
$$\nabla_{\zap A}: \mathcal{E} ({\mathbb V}_{+})\to \mathcal{E} (\Lambda^1\otimes {\mathbb V}_{+}),$$
and composition of this with the natural {\em Clifford multiplication} homomorphism
$$\Lambda^1\otimes {\mathbb V}_{+}\to {\mathbb V}_{-}$$
gives one  \cite{hitharm,lawmic} a 
spin$^c$ Dirac operator 
$$\drc_{\zap A}: \mathcal{E} ({\mathbb V}_{+})\to \mathcal{E} ({\mathbb V}_{-}).$$
Because of our special choice of spin$^c$ structure, this is an elliptic operator whose index, in complex-linear terms, was shown in 
Proposition \ref{toddgenus} to equal
the Todd genus
$$
\Ind_\CC (\drc_{\zap A}) = \dim_\CC \ker  (\drc_{\zap A})- \dim_\CC \ker \drc^*_{\zap A}) = \frac{(\chi + \tau )(M)}{4}
$$
despite the fact that we are working on the orbifold $(M,\Sigma, \beta )$ rather than on the original 
$4$-manifold $M$. 
%
%Of course, this immediately tells us that the index of the operator is, in real-linear terms, instead given by 
%$$
%\Ind_\RR (\drc_{\zap A}) = \dim_\RR \ker  (\drc_{\zap A})- \dim_\RR \ker \drc^*_{\zap A}) = \frac{(\chi + \tau )(M)}{2}~.
%$$

We will obtain our main results  by studying the Seiberg-Witten equations 
\begin{eqnarray} \drc_{\zap A}\Phi &=&0\label{drc}\\
 F_{A}^+&=&i \sigma(\Phi) ,\label{sd}\end{eqnarray}
where both the twisted spinor $\Phi$ and the unitary connection ${\zap A}$ are
treated as unknowns. Here
$F_{\zap A}^+$ is the self-dual part of the 
curvature  of ${\zap A}$, while 
 the natural real-quadratic map 
$\sigma : {\mathbb V}_+ \to \Lambda^+$
defined by $\sigma (\Phi) = -\frac{1}{2}\Phi\otimes \bar{\Phi}$ satisfies 
$$|\sigma (\Phi ) | = \frac{1}{2\sqrt{2}}|\Phi |^{2}.$$
These equations are non-linear, but they become an 
 elliptic first-order system 
once one imposes the `gauge-fixing' condition
\begin{equation}
\label{gauged}
d^* ({\zap A}-{\zap A}_0)=0
\end{equation}
relative to some arbitrary background connection ${\zap A}_0$;
this largely eliminates 
 the natural action of the `gauge group' 
of  automorphisms of the Hermitian line bundle 
$L\to M$, reducing it to the action of the $1$-dimensional group 
$U(1) \rtimes H^1(M, \ZZ)$ of harmonic maps $M\to S^1$.

In order to obtain an invariant that may in principle force the system (\ref{drc}-\ref{sd}) to have a solution, one
first considers the {\em perturbed} Seiberg-Witten equations
\begin{eqnarray} \drc_{\zap A}\Phi &=&0\label{drca}\\
i F_{A}^++\sigma(\Phi)&=&\eta ,\label{sda}\end{eqnarray}
where $\eta$ is a real-valued  self-dual $2$-form. We will say that 
the perturbation $\eta$ is {\em good} if 
$$\eta^H \neq 2\pi [\corb (L)]^+$$
where $\eta^H\in \mathcal{H}^+\subset \mathcal{E}^2(M,\Sigma, \beta)$ is the $L^2$-orthogonal projection of $\eta$ onto the harmonic $2$-forms
and where $[\corb (L)]^+\in \mathcal{H}^+ \subset H^2(M, \RR)$ is the cup-orthogonal projection into the deRham classes with self-dual representatives relative
to $g$. Whenever $\eta$ is a good perturbation, any solution $(\Phi, {\zap A})$ of   (\ref{drca}-\ref{sda}) 
is necessarily {\em irreducible}, in the sense that $\Phi\not\equiv 0$. Even in the presence of the gauge-fixing condition
\eqref{gauged}, this implies  that $U(1) \rtimes H^1(M, \ZZ)$ acts freely on the space of solutions.

 If $(g, \eta)$ is a pair consisting of a orbifold-smooth metric and a self-dual $2$-form which is a good perturbation with respect to $g$,
we will say that $(g, \eta)$ is a good pair. The set of good pairs 
  is connected if $b_+(M) > 1$. If $b_+(M)=1$, it instead consists of two connected components, called {\em chambers}. 
  Indeed, when $b_+(M)=1$, the intersection form on $H^2(M, \RR)$ is of Lorentz type, so that $H^2(M, \RR)$ can 
  be thought of as a copy of $b_2(M)$-dimensional Minkowski space. The set of vectors $\mathbf{a}\in H^2(M, \RR)$ with 
  $\mathbf{a}^2 > 0$ then becomes the set of time-like vectors, and has two connected components, and we may give this Minkowski space
  a  time-orientation by labeling one the 
 {\em future-pointing},  and the other {\em past-pointing}, time-like vectors. For any good pair $(g,\eta)$, the vector 
  $2\pi [c_1(L)]^+ - \eta$ is  a time-like vector. One chamber, which we will label $\uparrow$, consists  of those pairs $(g, \eta )$ 
  for which  $2\pi [c_1(L)]^+ - \eta$ is future-pointing; the other chamber, which we will call  $\downarrow$, consists  of those pairs $(g, \eta )$ 
  for which  $2\pi [c_1(L)]^+ - \eta$ is past-pointing.

  Equations  (\ref{drca}-\ref{sda}) imply the Weizenb\"ock formula
  \begin{equation}
\label{petweit}
  0 = 2\Delta |\Phi |^2 + 4 |\nabla_{\zap A} \Phi |^2 + 
s |\Phi |^2 +  |\Phi|^4 - 8 \langle \eta , \sigma (\Phi ) \rangle 
\end{equation}
  and from this one obtains the important $C^0$ estimate that any  solution must satisfy 
  $$|\Phi|^2 \leq \max (2\sqrt{2}|\eta |- s,0).$$
  By bootstrapping and the Rellich-Kondrakov theorem, it follows that the space of solution  solutions of
  (\ref{gauged}-\ref{drca}-\ref{sda}) of regularity $L^q_{k+1}$, modulo the restricted gauge action of  $U(1) \rtimes H^1(M, \ZZ)$, is compact in the $L^q_k$ topology. This argument also applies to the space of all solutions as $(g,\eta)$ varies over some compact family of pairs $(g,\eta)$. 
  
\begin{lem} \label{swindex} 
Let $(M,J_0)$ be a orbifold-smooth compact $4$-manifold with almost-complex structure, let 
$\Sigma \subset M$ be a compact embedded pseudo-holomorphic curve. Choose an integer $p\geq 2$, and  equip
the orbifold $(M, \Sigma, \beta)$, $\beta = 1/p$, with the the spin$^c$-structure 
obtained by lifting $J_0$ to an almost-complex structure $J$. For some orbifold-smooth 
metric $g$, let $(\Phi, {\zap A})$ be an  $L^q_{k+1}$ solution of (\ref{gauged}-\ref{drca}-\ref{sda}) for some $q> 4$, $k \geq 0$, and let
$$\mbox{\cyr L} : L^q_{k+1}(\mathbb{V}_+\oplus \Lambda^1) \to L^q_{k}(\mathbb{V}_-\oplus \Lambda^+\oplus \Lambda^0)$$
be the linearization of  (\ref{gauged}-\ref{drca}-\ref{sda}) at $(\Phi, {\zap A})$. Then $\mbox{\cyr L}$ is a Fredholm operator 
of  index $0$.
\end{lem}
\begin{proof}
The linearization $\mbox{\cyr L}$ differs from the  elliptic operator $\drc_{\zap A} \oplus d^+ \oplus d^*$ by lower order terms,
and so  has the same index. 
However, the kernel of 
$$d^+ \oplus d^* : L^q_{k+1} (\Lambda^1) \to L^q_{k}(\Lambda^+ \oplus \Lambda^0)$$ is the space of  harmonic 
$1$-forms $\mathcal{H}^1$, while its  cokernel is $\mathcal{H}^+ \oplus \mathcal{H}^0$, the space of self-dual harmonic $2$-forms direct sum
the constants.   The index of $\mbox{\cyr L}$ is therefore 
\begin{eqnarray*}
\Ind_\RR  (\mbox{\cyr L}) &=& 2 \Ind_\CC (\drc_{\zap A}) + (b_1 - b_+- b_0)(M) \\&=& 2\frac{(\chi + \tau ) (M)}{4} - \frac{(\chi + \tau ) (M)}{2} \\&=&0,
\end{eqnarray*}
using the value for  the index of $\drc_{\zap A}$  found in Proposition \ref{toddgenus}. 
\end{proof}

This index calculation gives rise to an  immediate generalization of the familiar definition of the Seiberg-Witten invariant 
of $4$-manifolds to the present orbifold context. 
The space of  solutions of (\ref{gauged}-\ref{drca}) with $L^q_{k+1}$ regularity and $\Phi\not\equiv 0$ is a Banach manifold $\mathcal{B}^q_{k+1}$, 
because, letting $L^q_k(0)\subset L^q_k$ be  the codimension-$1$ subspace of functions of integral $0$, because  (\ref{gauged}-\ref{drca}) 
may be interpreted as defining this  $\mathcal{B}^q_{k+1}$ to be the set of 
  $(\Phi, {\zap A})$ 
which are sent to 
$(\mathbf{0}, \mathbf{0})$  by a smooth map $[L^q_{k+1}(\mathbb{V}_+) -\mathbf{0}]\times L^q_{k+1} (\Lambda^1)\to L^q_{k}(\mathbb{V}_-) \times L^q_k(0)$
for which $(\mathbf{0}, \mathbf{0})$ is a regular value. Equation \eqref{sda} then defines a smooth  Fredholm map $\mathcal{B}^q_{k+1}\to L^q_k (\Lambda^+)$
which, by Lemma \ref{swindex}, has index $1$. For regular values $\eta \in L^q_k (\Lambda^+)$ of this map which are also good, 
we therefore conclude that the solutions of  (\ref{gauged}-\ref{drca}-\ref{sda}) form a $1$-manifold. Moreover, 
$U(1) \rtimes H^1(M, \ZZ)$ acts freely on this $1$-manifold, and the {\em moduli space} $\mathfrak{M}(g, \eta)$, obtained by dividing the space of solutions by 
the free action  of $U(1) \rtimes H^1(M, \ZZ)$,  is a $0$-manifold. Compactness therefore implies that the moduli space $\mathfrak{M}(g, \eta)$ is finite. 
Moreover, one obtains the following:

\begin{prop}
Let $(M,J_0)$ be a smooth compact $4$-manifold with almost-complex structure, let 
$\Sigma \subset M$ be a compact embedded pseudo-holomorphic curve. Choose some positive integer $p\geq 2$, and  equip
the orbifold $(M, \Sigma, \beta)$, $\beta = 1/p$, with the the spin$^c$-structure 
obtained by lifting $J_0$ to an almost-complex structure $J$. Then, for any metric $g$, 
and for a residual  set of good $\eta\in L^q_k(\Lambda^+)$  the Seiberg-Witten 
moduli space $\mathfrak{M}(g, \eta)$ is finite. If $b_+(M) > 1$, then, for the fixed spin$^c$ structure,
 any two such moduli spaces
$\mathfrak{M}(g, \eta)$ and $\mathfrak{M}(\tilde{g}, \tilde{\eta})$ 
are cobordant. If $b_+(M) = 1$, two such moduli spaces
$\mathfrak{M}(g, \eta)$ and $\mathfrak{M}(\tilde{g}, \tilde{\eta})$ are cobordant if, in addition, 
the good pairs 
$(g, \eta)$ and $(\tilde{g}, \tilde{\eta})$ belong to the same chamber.
\end{prop}

If $\mathfrak{c}$ denotes the spin$^c$ structure on the oriented orbifold $(M,\Sigma, \beta)$ induced by $J$,
and if $b_+(M) > 1$, the cobordism-invariance of the moduli space   allows us to 
  define the mod-$2$ Seiberg-Witten invariant of   $(M,\Sigma, \beta)$ 
by setting 
$$
SW_{\mathfrak{c}}(M,\Sigma, \beta) = \mbox{ Cardinality} (\mathfrak{M}(g, \eta)) \bmod 2
$$
for any good pair $(g,\eta )$ for which $\eta$ is a regular value of the relevant map. 
When $b_+(M) = 1$, this definition instead gives rise to  two $\mathbb{Z}_2$-valued invariants, $SW^\uparrow_{\mathfrak{c}}(M,\Sigma, \beta)$ and  
 $SW^\downarrow_{\mathfrak{c}}(M,\Sigma, \beta)$, where the arrow indicates whether the generic good pair$(g,\eta)$ at which we carry out 
 our mod-$2$ count belongs to the $\uparrow$ or the $\downarrow$ component of the space of good pairs. 
 
 If $b_+(M)> 1$ and $SW_{\mathfrak{c}}(M,\Sigma, \beta) \neq 0$, then the Seiberg-Witten equations (\ref{drc}-\ref{sd}) must
 have a solution for any orbifold metric $g$ on $(M,\Sigma, \beta)$.  Indeed, if $[\corb(L)]^+\neq 0$ for the given metric $g$, then $\eta=0$
 is a good perturbation; if there were no solution, it would also be a regular value, and we would obtain 
 a contradiction by performing a mod-$2$ count of solutions and concluding that $SW_{\mathfrak{c}}=0$.
 If  $[c_1(L)]^+= 0$, we still have a solution, albeit a reducible one, obtained by setting $\Phi\equiv 0$
 and arranging for $F_{\zap A}$ to be harmonic. Note that elliptic regularity actually guarantees
 that any solution $(\Phi, {\zap A})$ of (\ref{drc}-\ref{sd}-\ref{gauged}) is actually orbifold-smooth,
 assuming that the  metric $g$ is  orbifold-smooth, too.

 If $b_+(M)> 1$ and $SW^\downarrow_{\mathfrak{c}}(M,\Sigma, \beta) \neq 0$, the same argument shows that
  the Seiberg-Witten equations (\ref{drc}-\ref{sd}) must
 have a solution for any orbifold metric $g$ on $(M,\Sigma, \beta)$ such that  $[\corb (L)]^+$ is past-pointing. 
 Similarly, if $SW^\uparrow_{\mathfrak{c}}(M,\Sigma, \beta) \neq 0$, 
  the Seiberg-Witten equations (\ref{drc}-\ref{sd}) must
 have a solution for any orbifold metric $g$ on $(M,\Sigma, \beta)$ such that  $[\corb (L)]^+$ is future-pointing. 
 Here, it is worth recalling once again that, in the case that concerns us here,
 $$
 \corb (L) = c_1(M,J) + (\beta -1) [\Sigma]
 $$
 where $\beta = 1/p$.

 Of course, all this would be completely useless without some geometric criterion  to sometimes guarantee that 
 our  orbifold version of the  Seiberg-Witten invariant is  actually non-zero. Fortunately, such a criterion is implicit in 
 the work of  Taubes \cite{taubes}:
 
 \begin{thm} \label{crux} 
 Suppose the $4$-manifold $M$ admits a symplectic form $\omega_0$ which restricts to $\Sigma\subset M$
 as an area form. Choose an almost-complex structure $J_0$ which is compatible with $\omega_0$
 and which makes $\Sigma$  a pseudo-holomorphic curve.   Let $J$ be the  lift of $J_0$ to  $(M, \Sigma, \beta)$, 
 and let $\mathfrak{c}$ be the spin$^c$ structure on $(M, \Sigma, \beta)$ induced by $J$ as above. The following then hold:
 \begin{itemize}
 \item
 If $b_+(M) > 1$, then $SW_{\mathfrak{c}}(M,\Sigma, \beta) \neq 0$.
 \item
 If $b_+(M) = 1$, then $SW^\downarrow_{\mathfrak{c}}(M,\Sigma, \beta) \neq 0$, where 
  $H^2(M, \RR)$ has been time-oriented so that  $[\omega_0 ]$ is future-pointing.
  \end{itemize}
 \end{thm}
 \begin{proof} Choose an orbifold metric $g$ on $(M,\Sigma, \beta)$ which is adapted to the lift $\omega$ of $\omega_0$ to $(M, \Sigma, \beta)$. 
 Let $t\gg  0$ be a large positive real constant, and let ${\zap A}_0$ be the Blair connection \cite{blair} on $L=\Lambda^{2,0}_J$. 
 Then  $\Phi_0 = (\sqrt{t}, 0)\in \mathcal{E}^{0,0}\oplus \mathcal{E}^{0,2}$ solves \cite{lsymp,taubes} 
 $$\drc_{{\zap A}_0}\Phi =0.$$
 Hence $(\Phi_0, {\zap A}_0)$ 
 can be viewed as a solution of the perturbed Seiberg-Witten equations (\ref{drca}-\ref{sda}) for an appropriate choice
 of perturbation $\eta$,
 namely the self-dual form given, in the notation of \eqref{bla1},  by 
 $$
 \eta = \frac{t}{4}\omega + iF^+_{{\zap A}_0}  = \left(2t+ s+s^* \right) \frac{\omega}{8} + W_+(\omega)^\perp.
 $$
If $t$ is sufficiently large,  Taubes \cite{taubes} then shows that, for this choice of $\eta$,  {\em any} solution
$(\Phi , {\zap A})$ of (\ref{drca}-\ref{sda}) of the system is then gauge equivalent to $(\Phi_0 , {\zap A}_0)$, 
and that 
$\eta$ is moreover a regular value of the Seiberg-Witten map.
(While Taubes assumes that  he is working on a manifold rather than an orbifold, all of his arguments  go through in the present 
context without change.) Since $\mathfrak{M}(g,\eta)$ consists of a single point, our mod-$2$ count therefore gives a non-zero answer. 
When $b_+(M) > 1$, this shows that $SW_{\mathfrak{c}}(M,\Sigma, \beta) \neq 0$. When  $b_+(M) = 1$, it instead shows that the
mod-$2$ Seiberg-Witten invariant is non-zero for some chamber, and since $[\omega ]\cdot (2\pi [\corb (L)]^+ - \eta) = - t[\omega ]^2/4 < 0$,
the relevant chamber is the past-pointing one relative to $[\omega]=[\omega_0]$.
 \end{proof}

\section{Curvature Estimates}

The Seiberg-Witten  equations (\ref{drc}-\ref{sd}) imply the {Weitzenb\"ock formula}
 \begin{equation}
 0=	2\Delta |\Phi|^2 + 4|\nabla\Phi|^2 +s|\Phi|^2 + |\Phi|^4 ,	
 	\label{wnbk}
 \end{equation}
where $\nabla = \nabla_{\zap A}$.
This leads to  some rather surprising curvature estimates;  cf. \cite{witten,lpm,lno,lric,lsymp}. 
We begin by proving some estimates in terms of the self-dual part $\corb (L)^+\in \mathcal{H}^+_g$ of $\corb (L)$.

\begin{lem}\label{best}
Let $(M, J_0)$ be an almost-complex $4$-manifold, and let $\Sigma\subset M$ be a compact embedded
pseudo-holomorphic curve. Choose some integer $p\geq 2$, and let $(M,\Sigma, \beta)$, $\beta = 1/p$, be the orbifold 
obtained by declaring the total angle around $\Sigma$ to be $2\pi\beta$.  Let $\mathfrak{c}$ be the
spin$^c$ structure induced by the lifting $J_0$ to an almost-complex structure $J$ on $(M,\Sigma, \beta)$.
\begin{itemize}
\item
 If $b_+(M) > 1$, and if $SW_{\mathfrak{c}}(M,\Sigma,\beta) \neq 0$, then the scalar curvature $s_g$ of any orbifold metric $g$ on $(M,\Sigma, \beta)$ satisfies
 \begin{equation}
\label{scalar}
\int_{M}s_{g}^{2}d\mu_{g} \geq 32\pi^{2} [\corb(L)^{+}]^{2} ,
\end{equation}
with equality if and only if $g$ is an orbifold K\"ahler metric of constant negative scalar curvature which is 
compatible with a complex structure on  $(M,\Sigma, \beta)$ with $\corb = \corb (L)$ and with 
$(\corb )^+$ a negative multiple of the K\"ahler class $[\omega ]$. 
\item
If $b_+(M)=1$, and if $SW^\downarrow_{\mathfrak{c}}(M,\Sigma,\beta)  \neq 0$
for a fixed time orientation for $H^2(M,\RR)$,
then the same conclusion holds for any $g$ such that  $\corb(L)^{+}$ is past-pointing. 
\end{itemize}
\end{lem}
\begin{proof} Our hypotheses guarantee that there is a solution $(\Phi, {\zap A})$ of  (\ref{drc}-\ref{sd})  for the given metric and spin$^c$ structure.
Integrating (\ref{wnbk}), we have
$$0= \int [ 4|\nabla \Phi |^2 + s|\Phi|^2 + |\Phi|^4 ] d\mu , $$
and it follows that 
$$\left(\int s^2 d\mu  \right)^{1/2}\left(\int |\Phi |^4 d\mu  \right)^{1/2} \geq \int (-s) |\Phi|^2 d\mu \geq \int |\Phi|^4 d\mu ,$$
so that 
$$
\int s^2 d\mu \geq \int |\Phi |^4 d\mu  = 8 \int |F_{\zap A}^+|^2 d\mu .
$$ 
 Since $iF_{\zap A}$ belongs to the deRham class $2\pi \corb (L)$, 
 we moreover have 
 $$\int |F_{\zap A}^+|^2 d\mu \geq 4\pi^2 [\corb(L)^+]^2$$
by Lemma \ref{compare}. Together, these two inequalities imply \eqref{scalar}.

If equality holds in \eqref{scalar}, the above argument shows that $\nabla \Phi \equiv 0$,  and that    $|\Phi|^2$ and $(-s)$ 
are  both non-negative constants.  If $\corb(L)^+\neq 0$, the constants $(-s)$ and  $|\Phi|^2$ must both be positive, and the parallel 
 section $\Phi$ of $\mathbb{V}_+$ is non-zero.  This implies that 
the self-dual $2$-form $\sigma (\Phi)$ is parallel and non-zero, so that   $g$ must K\"ahler. 
 This shows that $g$ is a constant-scalar-curvature
K\"ahler metric with $s< 0$. Moreover, $\Phi\otimes \Phi$ is a non-zero section of $\Lambda^{2,0}\otimes L$, so this K\"ahler metric
has $\corb = \corb (L)$. 

It only remains to see what happens when  $b_+(M) > 1$, $SW_{\mathfrak{c}}(M,\Sigma,\beta) \neq 0$, and both sides of \eqref{scalar}
vanish. In this case, the non-vanishing of the Seiberg-Witten invariant implies that there cannot be a positive-scalar-curvature orbifold metric
on $(M,\Sigma,\beta)$, 
because  otherwise there would exist solutions of  the perturbed Seiberg-Witten equations (\ref{drca}-\ref{sda}) for a 
good perturbation $\eta$ with $s-2\sqrt{2}|\eta | > 0$, leading to a contradiction via the perturbed Weitzenb\"ock formula \eqref{petweit}. 
On the other hand, we have assumed that both sides of \eqref{scalar}
vanish, so $g$ is scalar flat; and if it were not Ricci-flat, there  would be a nearby  metric of positive scalar curvature \cite{bes,bouric}.
Thus $g$ must be Ricci-flat, and so, in particular, Einstein. By \eqref{urbi},  this implies that $[\corb (L)]^2\geq 0$, with equality iff $g$ is scalar-flat and anti-self-dual.  
But since we have assumed that both sides of \eqref{scalar}
vanish, we have $[\corb (L)^+]^2=0$, so that  
$$
 [\corb (L)]^2 =[\corb (L)^+]^2-|[\corb (L)^-]^2| \leq 0.
$$
Hence $[\corb (L)]^2= 0$, and our scalar-flat metric  $g$ is   also anti-self-dual. 
 But  this implies $b_+(M)=1$ by Proposition \ref{exclude}. Thus the case under discussion never occurs, and we are done. 
\end{proof}

We will also need an  analogous curvature estimate for a mixture of the scalar and Weyl curvatures.
To  this end, let $f > 0$ be a postive orbifold-smooth  function, and 
consider the {\em rescaled Seiberg-Witten equations} \cite{lebsurv2} 
\begin{eqnarray} D_{A}\Phi &=&0\label{rsdrc}\\
 F_{A}^+&=&if\sigma (\Phi) .\label{rssd}\end{eqnarray}
 which, for reasons related to the conformal invariance of the  Dirac operator \cite{hitharm,lawmic,pr2},   are simply a disguised form of the Seiberg-Witten 
 equations for the conformally related metric $f^{-2}g$. When $b_+(M) > 1$, the non-vanishing of $SW_\mathfrak{c}(M)$ is therefore enough 
 to guarantee that (\ref{rsdrc}--\ref{rssd})  have a solution for any $g$ and any $f$. Similarly, when $b_+(M)=1$, 
 the non-vanishing of  $SW_\mathfrak{c}^\downarrow(M)$ suffices to guarantee that, for any $f> 0$, (\ref{rsdrc}--\ref{rssd})  must have a solution 
 for any metric $g$ such that $\corb (L)^+$ is past-pointing. 
 Will now exploit  the fact that  (\ref{rsdrc}--\ref{rssd})  imply the Weitzenb\"ock formula 
  \begin{equation}
 0=	2\Delta |\Phi|^2 + 4|\nabla\Phi|^2 +s|\Phi|^2 + f|\Phi|^4 ,	
 	\label{rswnbk}
 \end{equation}
to obtain the desired curvature estimate.

 \begin{lem}\label{upscale}
 Let $(M,\Sigma, \beta)$ and $\mathfrak{c}$ be as in Lemma \ref{best}.
 \begin{itemize}
 \item 
  If $b_+(M) > 1$, and if $SW_{\mathfrak{c}}(M,\Sigma,\beta) \neq 0$, 
 then the scalar curvature 
$s_g$ and self-dual Weyl curvature $(W_{+})_g$ of any orbifold metric $g$ on 
$(M,\Sigma, \beta)$ 
satisfy
 $$\int_M (s-\sqrt{6}|W_+|)^2d\mu_g \geq 72\pi^2 [\corb(L)^{+}]^{2} ,
$$
 with equality if and only if  $g$  is a saturated 
 almost-K\"ahler  metric,  in the sense of Definition \ref{fatso},  with $\corb = \corb (L)$, and with 
  $(\corb)^+$  a negative multiple of the symplectic class $[\omega ]$. 
  \item
 If $b_+(M)=1$, and if $SW^\downarrow_{\mathfrak{c}}(M,\Sigma,\beta)  \neq 0$
for a fixed time orientation for $H^2(M,\RR)$,
then the same conclusion holds for any $g$ for which $\corb(L)^{+}$ is past-pointing. 
\end{itemize}
  \end{lem}
 \begin{proof} 
For any   smooth  function $f> 0$ on $M$, our hypotheses guarantee that 
 (\ref{rsdrc}--\ref{rssd}) must 
admits a solution  $(\Phi ,  A)$. We now proceed, {\em mutatis mutandis}, as in \cite{lebeta}. 
Multiplying \eqref{rswnbk} by $|\Phi |^2$ and integrating, 
this solution must then  satisfy 
$$0\geq \int_M \left( 4 |\Phi |^2 |\nabla_{\zap A}\Phi|^2 +s|\Phi|^4 + f|\Phi|^6\right) d\mu~.$$
On the other hand, the self-dual $2$-form $\psi = 2\sqrt{2} \sigma (\Phi )$
automatically satisfies 
$$
|\Phi |^{4} =	|\psi |^{2}   ,\qquad 
	   4 |\Phi |^{2}|\nabla_{\zap A} \Phi |^{2}  \geq  |\nabla \psi |^{2}  , 
$$
so it follows that 
 $$0\geq \int_M\left(|\nabla\psi  |^{2} + s|\psi |^{2}
 + f|\psi |^{3}\right) d\mu . $$
However, inequality  (\ref{part}) also tells us that
$$
\int_{M} |\nabla \psi |^{2}d\mu \geq 
\int_M\left(-2\sqrt{\frac{2}{3}}|W_+|-\frac{s}{3}\right)|\psi |^{2} d\mu , 
$$
and combining these facts yields 
$$
0\geq \int_M\left[  \left( s-\sqrt{6} |W_+|\right) |\psi |^{2} 
 + \frac{3}{2} f|\psi |^{3}\right] d\mu .
$$
Repeated use of  H\"older inequalities gives us 
$$
   \left(\int_M f^4d\mu\right)^{1/3}   \left( \int_M \left|s-\sqrt{6}|W_+|\right|^3
   f^{-2}d\mu\right)^{2/3} \geq 
  \int_M \left( \frac{3}{2} f |\psi  |\right)^2 d\mu ,   $$
  and application of Lemma \ref{compare} to the right-hand-side then yields
  \begin{equation}
\label{pinpoint}
 \left(\int_M f^4d\mu\right)^{1/3}   \left( \int_M \left|s-\sqrt{6}|W_+|\right|^3
   f^{-2}d\mu\right)^{2/3}    ~\geq 72\pi^2 [\corb(L)^+]^2~
\end{equation}
for any smooth positive function $f$. Choosing a sequence $f_j$ of smooth positive functions such that 
 $$ f_j \searrow \sqrt{\left|  s-\sqrt{6}|W_+|   \right|}$$
we  then have 
$$
\int_M f_j^4 d\mu \geq  72\pi^2 [\corb(L)^+]^2, 
$$
and taking the limit as $j\to \infty$  therefore gives us \begin{equation}
\label{voila}
\int_M \left(s-\sqrt{6}|W_+|\right)^{2}d\mu  \geq  72\pi^2 [\corb(L)^+]^2~
\end{equation}
as  desired. 
 
If equality holds, we must have $$s-\sqrt{6}|W_+| = \mbox{constant}$$
because $g$ minimizes the left-hand-side in its conformal class. 
If this constant were zero, we would then have $s= \sqrt{6}|W_+|\equiv 0$, 
since $g$ would otherwise be conformal to  a metric of positive scalar curvature.
However, we would also have $[\corb (L)^+]^2=0$, and our hypotheses exclude this because 
Proposition \ref{exclude} forces $b_+(M)=1$. Thus 
$$f =  \sqrt{\left|  s-\sqrt{6}|W_+|   \right|} $$
is a positive function, and 
we would necessarily obtain equality in \eqref{pinpoint} for this choice of positive function; and 
our use of Lemma \ref{compare} in the previous argument would  force $\psi$ to be self-dual harmonic. Moreover,  our use of the 
H\"older inequality guarantees that $|\psi|$ must be a non-zero constant. In particular, $\Phi$ is everywhere non-zero, and $\Phi\otimes \Phi$ 
now defines a non-zero section of $\Lambda^{2,0}_J \otimes L$, so that $L$ is actually the anti-canonical line bundle of $J$. 
Finally, the fact that \eqref{part} must be saturated guarantees that 
$\omega$ belongs to the lowest eigenspace of $W_+$, and  the fact that $s-\sqrt{6}|W_+|$ is a negative constant now translates
into the statement that $s+s^*$ is a negative constant. Thus $g$ is a saturated almost-K\"ahler metric with respect to the symplectic form $\omega$. 
 \end{proof}

While these Lemmas do provide us with  interesting lower bounds for certain curvature quantities, the right-hand-sides of the
inequalities still depend on the decomposition $H^2(M, \RR) = \mathcal{H}^+_g\oplus \mathcal{H}^-_g$, and so are
still somewhat metric-dependent. Fortunately, however, they do imply other estimates which do not suffer from this short-coming: 
 
  \begin{thm} \label{geometrize} 
  Let $(X, J_0)$ be a compact  almost-complex $4$-manifold, and 
   let $\Sigma \subset X$
  be a compact embedded pseudo-holomorphic curve.
  For some  integer $p\geq 2$, set $\beta = 1/p$, 
   let $(X, \Sigma , \beta)$ be the  orbifold obtained from $X$ by declaring the total angle around
  $\Sigma$ to be $2\pi\beta$, and let $(M,\Sigma,\beta )$ be the orbifold obtained by blowing up 
  $(X, \Sigma , \beta)$ at $\ell\geq 0$ points away from $\Sigma$. Give  $(M,\Sigma,\beta )$ 
  an almost-complex structure $J$ by first lifting $J_0$ to the orbifold $(X, \Sigma , \beta)$ , and then to its blow-up 
  $(M,\Sigma,\beta )$; and  let  $\mathfrak{c}$ be the  spin$^c$ structure on $(M,\Sigma,\beta )$ induced by $J$. 
  If  If $b_+(M) > 1$, suppose that $SW_\mathfrak{c} (M,\Sigma, \beta)\neq 0$; if $b_+(M) =1$, instead suppose that 
   $SW_\mathfrak{c}^\downarrow (M,\Sigma, \beta)\neq 0$, and that $\corb (X)$ is time-like or null past-pointing. 
 Then the curvature of
    any orbifold Riemannian metric $g$ on $(M,\Sigma,\beta )$ satisfies 
  \begin{eqnarray*}
\int_M s^2 d\mu &\geq& 32\pi^2 (\corb )^2(X, \Sigma, \beta ) \\
\int_M (s-\sqrt{6}|W_+|)^2 d\mu &\geq & 72\pi^2 (\corb )^2(X, \Sigma, \beta )
\end{eqnarray*}
and both inequalities are strict  unless $\ell=0$ and $g$ is an 
 orbifold K\"ahler-Einstein metric  on $(X,\Sigma,\beta )$ with negative scalar curvature.
  \end{thm}
  \begin{proof}
  If $(\corb )^2(X, \Sigma, \beta ) < 0$, there is nothing to prove, so we may assume from the outset that 
  $(\corb )^2(X, \Sigma, \beta ) \geq 0$. When $b^+(X)=1$, this means that $\corb (X, \Sigma, \beta )$ 
  is either time-like or null. Our hypotheses moreover imply that it must be  past-pointing, and s its pull-back to 
  $M$ is therefore also  time-like or null  and past-pointing with respect to the given time-orientation. 
  Here we are using the fact that 
  $$
  H^2(M, \RR) = H^2(X,\RR) \oplus \underbrace{H^2(\overline{\CP}_2,\RR) \oplus \cdots \oplus H^2(\overline{\CP}_2,\RR)}_\ell
  $$
  by applying Mayer-Vietoris to $M\approx X \# \ell \overline{\CP}_2$; the summands are mutually orthogonal with respect to the
  intersection form, and the intersection form restricts to each of these subspaces as the intersection form of the 
  corresponding topological building block. 

  If $E_1 , \ldots , E_\ell\in H^2 (M)$  are standard  generators for the 
 $\ell$ copies of $H^2(\overline{\CP}_2,\RR)$ obtained by taking Poincar\'e duals of the $\ell$ projective lines $S_1, \ldots , S_\ell$ introduced
 by blow-up, then the given spin$^c$ structure $\mathfrak{c}$ has Chern class 
  $$
 \corb (M, \Sigma, \beta ) = \corb (X, \Sigma , \beta ) - E_1-\cdots -E_\ell~.
 $$
 However, there are self-diffeomorphisms \cite{FM,lno} of $M$ which are the identity outside a tubular neighborhood of a given $S_j$, but
 which send $E_j$ to $-E_j$. Applying these to $\mathfrak{c}$ give us $2^{\ell}$ different spin$^c$ structures, with 
   $$
\widetilde{\corb} (M, \Sigma, \beta ) = \corb (X, \Sigma , \beta ) \pm E_1\pm \cdots \pm E_\ell~
 $$
 for any desired choice of signs; 
 and when $b_+(M) = b_+(X)=1$, these diffeomorphisms moreover preserve the given time orientation of $M$. 
 Given a metric $g$ on $M$, let us  now choose  $\epsilon_j= \pm E_j$ so that  
  $$[\corb (X, \Sigma , \beta )]^+ \cdot  \epsilon_j^+= [\corb (X, \Sigma , \beta )]^+ \cdot  \epsilon_j \geq 0$$
  and then observe that, when $b_+(M) =1$, $[\widetilde{\corb}]^+$ is  past-pointing  for this spin$^c$ structure and this metric. 
  Moreover, we then have 
  \begin{eqnarray*}
\Big( [\widetilde{\corb} (M, \Sigma, \beta )]^+\Big)^2 &=&  \Big( [{\corb} (X, \Sigma, \beta )]^+\Big)^2+ 2[\corb (X, \Sigma , \beta )]^+ \cdot  \sum_j \epsilon_j^+
+ (\sum_j \epsilon_j^+)^2
\\  &\geq& \Big( [{\corb} (X, \Sigma, \beta )]^+\Big)^2 
\\   &\geq &  \Big( {\corb} (X, \Sigma, \beta )\Big)^2 ~
\end{eqnarray*}
so 
  Lemmas \ref{best} and \ref{upscale} immediately give us the two promised curvature inequalities.

  If equality were to hold in either of these inequalities, we would necessarily have ${\corb} (X, \Sigma, \beta )^-=0$, and 
  $ [\corb (X, \Sigma , \beta )]^+ \cdot \epsilon_j=0$ for $j=1,\ldots , \ell$. In particular, if we replace  $\epsilon_j$ 
  with $-\epsilon_j$ for each $j=1,\ldots , \ell$, we would obtain a second spin$^c$ structure which also saturated the relevant
  inequality. But  these two spin$^c$ structures have identical $(\corb)^+$ with respect to the given metric. 
  Lemma \ref{best} or \ref{upscale} then tell us that 
   $g$ is  almost-K\"ahler  with respect to 
   two symplectic forms $\omega$ and $\tilde{\omega}$ which are harmonic representatives of the same negative multiple
  of $(\corb)^+$ and  whose Chern classes differ by $2 \sum_j \epsilon_j$. The latter of
  course implies that $\omega= \tilde{\omega}$, so that we must have $\sum_{j=1}^\ell \epsilon_j=0$, and this  can only happen if  $\ell =0$. 
  
  Thus equality implies that $M=X$,  and that $(\corb)^-=0$. Moreover, we know that 
  $g$ is a saturated almost-K\"ahler metric, and that the curvature $F_{\zap B}$ of the Blair connection 
   is the harmonic representative of $-2\pi i \corb$. In particular, in our present situation,   $F_{\zap B}$ is a self-dual form. However, 
by specializing  (\ref{bla1}--\ref{bla2}) to the saturated case, the curvature of the Blair connection of such a metric is given by 
 $$
 iF_{\zap B}^+ = \frac{s+s^*}{8}\omega, \qquad 
iF_{\zap B}^- =   \frac{s-s^* }{8}\hat{\omega}+ \mathring{\varrho},
 $$  
   where 
$s+s^*$ is a non-positive constant,  $\mathring{\varrho}$ encodes the $J$-invariant piece of the trace-free Ricci curvature $\mathring{r}$,
and where   the bounded anti-self-dual $2$-form 
 $\hat{\omega}\in \Lambda^-$
 satisfies 
 $|\hat{\omega}|\equiv \sqrt{2}$, 
but
 is  defined 
  only on the open set where $s^* - s\neq 0$. 
  In the present situation, however,  we also know that  $F_{\zap B}^-=0$, so  $\mathring{\varrho}=  (s^* -s) \hat{\omega}/8$, and 
  \begin{equation}
  \label{tip}
|\mathring{r}|^2 \geq \frac{(s^*-s)^2}{16}
\end{equation}
  everywhere,  with equality precisely at the points where 
the Ricci tensor $r$ is $J$-invariant. On the other hand, the algebraic constraint on  $W_+$ imposed by 
Definition \ref{fatso} implies that 
 $$|W_+|^2 = \frac{(3s^*-s)^2}{96}.$$
 Thus, the Gauss-Bonnet-type formula \eqref{urbi} tells us that 
 \begin{eqnarray*}
4\pi^2 (\corb )^2 (M,\Sigma, \beta)  & = & \int_M \left(\frac{s^2}{24}+ 2 |W_+|^2 - \frac{|\mathring{r}|^2}{2}\right)d\mu  \\
 & = &  \int_M   \left(\frac{s^2}{24}+  \frac{2(3s^*-s)^2}{96}- \frac{|\mathring{r}|^2}{2}\right)d\mu \\
  & \leq &    \int_M \left(\frac{s^2}{24}+  \frac{2(3s^*-s)^2}{96}- \frac{(s^*-s)^2}{32}\right)d\mu \\
   & = &  \frac{1}{32}  \int_M  \left(s^2-2ss^* +5(s^*)^2\right)d\mu ,
\end{eqnarray*}
with equality iff equality holds in \eqref{tip}. On the other hand, since $F_{\zap B}=F_{\zap B}^+$,  
\begin{eqnarray*}
4\pi^2 (\corb )^2 (M,\Sigma, \beta) 
 & = &  \int_M  (\frac{s+s^*}{8}\omega )\wedge  (\frac{s+s^*}{8}\omega ) \\
  & =&  \frac{1}{32}   \int_M  \left(s^2+2ss^* +(s^*)^2\right)d\mu ~.
\end{eqnarray*}
We therefore know that 
$$
\int_M  \left(s^2-2ss^* +5(s^*)^2\right)d\mu \geq   \left(s^2+2ss^* +(s^*)^2\right)d\mu , 
$$
or in other words that 
\begin{equation}
\label{gizmo}
\int_M  4s^*(s^*-s)  d\mu \geq  0 ,
\end{equation}
with    equality iff  equality holds in \eqref{tip}.
However, since $(M,g,\omega)$ is saturated,  $s^*+s$  is a negative constant,
and $W_+(\omega , \omega )  \leq 0$; hence 
$s^*\leq s/3$, and  $s^* \leq (s+s^*)/4$ is therefore negative everywhere. 
Since 
 $s^*-s = |\nabla \omega |^2 \geq 0$  on  any almost-K\"ahler  orbifold, we thus have 
$$s^* (s^*-s) \leq 0$$
everywhere on $M$, 
with equality only at points where $s=s^*$. 
The  inequality (\ref{gizmo}) therefore  implies 
 $$|\nabla \omega |^2 = s^*-s\equiv 0,$$ so that   $(M,g,\omega)$ is K\"ahler.
But  equality in (\ref{gizmo}) implies that equality also holds in \eqref{tip}, so 
$|\mathring{r}|^2 \equiv (s^*-s)^2/16$, and hence  
 $\mathring{r} \equiv 0$. Thus 
 $(M,g)$ is  K\"ahler-Einstein, with negative scalar curvature, as promised. 
   \end{proof}

 Theorem \ref{oui} is now an immediate consequence of Theorems \ref{crux} and  \ref{geometrize}, 
in conjunction with  the  explicit formula 
 $$\corb (X,\Sigma , \beta)  = c_1 (X) + (\beta -1) [\Sigma ]$$
 for the orbifold Chern class provided by  Lemma \ref{plf}. 
 
\section{Inequalities for Einstein Metrics} \label{concrete}

With the results of the previous section in hand, we are finally in a position to  prove Theorem \ref{non}. For this purpose, the key observation is the following:

  \begin{prop} \label{tinker} Let $X,M, \Sigma, \beta$ and $\mathfrak{c}$ be as in Theorem \ref{geometrize}. 
 Then the curvature of
    any orbifold  metric $g$ on $(M,\Sigma,\beta )$ satisfies 
    \begin{equation}
\label{ami}
\frac{1}{4\pi^2}\int_M \left(\frac{s^2}{24} + 2|W_+|^2\right)d\mu > \frac{2}{3} \left[\corb (X,\Sigma, \beta)\right]^2.
\end{equation}
  \end{prop}
\begin{proof}
Since there is otherwise nothing to prove, we may henceforth assume that $\left[\corb (X,\Sigma, \beta)\right]^2\geq 0$. With this proviso, 
the second inequality of Theorem  \ref{geometrize} tells us that 
$$
\Big\| s- \sqrt{6} |W_+|\Big\|_{L^2} \geq \left( 72\pi^2 \left[\corb (X,\Sigma, \beta)\right]^2\right)^{1/2}
$$
and the triangle inequality therefore yields 
$$
\| s \|_{L^2}+ \sqrt{6}~ \|W_+ \|_{L^2} \geq \left( 72\pi^2 \left[\corb (X,\Sigma, \beta)\right]^2\right)^{1/2},
$$
with equality only if  $g$ is K\"ahler-Einstein, with negative scalar curvature. However, the left-hand side can be interpreted as a
dot product
$$(1 , {\frac{1}{\sqrt{8}}}) \cdot \left(\|s\|, 
 \sqrt{48} \|W_{+}\|\right) 
$$
in $\RR^2$, and the Cauchy-Schwarz inequality thus tells us that 
$$
\left(1+\frac{1}{8}\right)^{1/2} \left( \|s\|^2_{L^2} + 48 \|W_{+}\|^2_{L^2}\right)^{1/2} \geq \left( 72\pi^2 \left[\corb (X,\Sigma, \beta)\right]^2\right)^{1/2},
$$
with equality only if $(1 , {\frac{1}{\sqrt{8}}})\propto \left(\|s\|, 
 \sqrt{48} \|W_{+}\|\right) $, and  $g$ is K\"ahler-Einstein with negative scalar curvature. However since $\|s\|= \sqrt{24}\|W_+\|$ for any K\"ahler metric, equality can never actually occur. 
  Squaring both sides and dividing by $108\pi^2$ thus yields the promised strict
 inequality. 
\end{proof}

\begin{cor} \label{tailor} 
Let $X,M, \Sigma, \beta$ and $\mathfrak{c}$ be as in Theorem \ref{geometrize},
so that $M\approx X\# \ell \overline{\CP}_2$. 
If  $(M,\Sigma, \beta)$ admits an orbifold Einstein metric $g$, then 
    $$\ell <   \frac{1}{3} ( \corb)^2 (X,\Sigma, \beta).$$
\end{cor}

\begin{proof} When the orbifold metric $g$ on $(M,\Sigma, \beta)$ is Einstein, \eqref{urbi} asserts that 
$$\frac{1}{4\pi^2}\int_M \left(\frac{s^2}{24} + 2|W_+|^2\right)d\mu = ( \corb)^2 (M,\Sigma, \beta).$$
However,   notice that 
\begin{eqnarray*}
( \corb)^2 (M,\Sigma, \beta) &=& (c_1 (M) + (\beta -1) [\Sigma ])^2 \\
&=& c_1^2 (M) + 2(\beta -1) ~c_1(M)\cdot [\Sigma ] + (\beta -1)^2 [\Sigma ]^2 \\
&=& c_1^2 (X) -\ell + 2(\beta -1) ~c_1(X)\cdot [\Sigma ] + (\beta -1)^2 [\Sigma ]^2 \\
&=& ( \corb)^2 (X,\Sigma, \beta) -\ell . 
\end{eqnarray*}
If  $(M,\Sigma, \beta)$ admits  an orbifold Einstein metric $g$,   \eqref{ami} therefore says that 
$$
 ( \corb)^2 (X,\Sigma, \beta) -\ell >  \frac{2}{3} ( \corb)^2 (X,\Sigma, \beta)
$$
and it follows that $ ( \corb)^2 (X,\Sigma, \beta)/3 > \ell$, as claimed. 
\end{proof} 

 Assuming the non-vanishing of a suitable Seiberg-Witten invariant,   the contrapositive of Corollary \ref{tailor}  is that 
$(M, \Sigma, \beta )$ does not admit an orbifold Einstein metric if 
$$
\ell \geq \frac{1}{3} ( \corb)^2 (X,\Sigma, \beta). 
$$ 
On the other hand, Theorem \ref{crux} gives a geometric criterion sufficient  to guarantee that 
 the relevant Seiberg-Witten invariants are indeed non-zero. Moreover, 
Proposition \ref{reggie} asserts that any edge-cone metric on $M$ of edge-cone
angle $2\pi \beta$, $\beta = 1/p$, is actually orbifold Einstein  metric on $(M, \Sigma, \beta )$. 
Given  the explicit formula
 $$\corb (X,\Sigma , \beta)  = c_1 (X) + (\beta -1) [\Sigma ]$$
 provided by  Lemma \ref{plf}, 
 Theorem \ref{non}  therefore follows. 
 
 \bigskip 
 
 Let us now illustrate Theorem \ref{non} via some concrete examples. 
 
 \begin{xpl} Let $X\subset \CP_3$ be a quadric surface, let $Y\subset \CP_3$ be a cubic surface which meets
 $X$ transversely, and let $\Sigma = X\cap Y$ be the genus $4$ curve in which they intersect. Let $M\approx X\# \overline{\CP_2}$ be the 
blow-up of $X$ at a point not belonging to $\Sigma$. Then 
$$(\corb (X,\Sigma , \beta ))^2 = \left( c_1 (X) +(\beta -1) [\Sigma ]\right)^2 = 2(1- 3\beta)^2$$
so  $(\corb (X,\Sigma , \beta ))^2/3 < 1= \ell$ for all $\beta = 1/p$, $p\geq 2$. On the other hand, 
  $(\corb (X,\Sigma , \beta ))^2 < 1$ provided that $2\leq p \leq 10$. 
It follows that $(M,\Sigma, \beta )$ does not admit orbifold Einstein metrics for any $p$;  consequently,  there exist no Einstein edge-cone metrics
on $(M,\Sigma)$  of cone angle $2\pi \beta$ for any $\beta = 1/p$, $p\geq 2$ an integer. Indeed, when $p> 3$, inequality \eqref{common} holds
if we take  $\omega_0$ to be the restriction of the Fubini-Study K\"ahler form to $X$, so in this range the claim follows from Theorem \ref{non}. 
Proposition \ref{htal} rules out the existence of Einstein edge-cone metrics for the two cases $p=2$ or $3$ not covered by this argument, 
and moreover also applies to an interval of  real-valued $\beta$, but does not suffice to prove the assertion when the cone angle is small. 
  \end{xpl}

\begin{xpl} Let $Y\subset \CP_3$ once again be a cubic surface, and let $\Sigma =  Y\cap X$ once again   be its intersection with a generic quadric.
This time, however, let us instead consider orbifold versions of the 
the manifold $M\approx Y\# \overline{\CP}_2$ obtained by blowing up $Y$ at a point not belonging to $\Sigma$. Then 
$$(\corb (Y,\Sigma , \beta ))^2 = \left( c_1 (Y) +(\beta -1) [\Sigma ]\right)^2 = 3(1- 2\beta)^2$$
so once again $(\corb (Y,\Sigma , \beta ))^2/3 < 1= \ell$ for all $\beta = 1/p$, $p\geq 2$, while $(\corb (Y,\Sigma , \beta ))^2 < 1= \ell$
when $p= 2$, $3$, or $4$. Once again, there are no orbifold Einstein metrics on $(M,\Sigma, \beta )$ 
for any $p$, and  consequently no Einstein edge-cone metrics
on $(M,\Sigma)$  of cone angle $2\pi \beta$ for any $\beta = 1/p$, $p\geq 2$ an integer. This time, \eqref{common}
holds for all $p> 2$, so we only need to use Proposition \ref{htal} for one single case. 
\end{xpl}

\begin{lem}
\label{soldier}
For some integer $p\geq 2$, set $\beta=1/p$, and let 
 $(M,\Sigma, \beta)$ and $\mathfrak{c}$ be as in Lemma \ref{best}. 
If $b_+(M) > 1$,  suppose  that $SW_{\mathfrak{c}}(M,\Sigma,\beta) \neq 0$; if $b_+(M) = 1$, instead suppose that $\corb (L)$ is time-like past-pointing, and that 
$SW^\downarrow_{\mathfrak{c}}(M,\Sigma,\beta)  \neq 0$.  If $g$ is an orbifold Einstein metric on $(M,\Sigma,\beta )$, then 
   its scalar and self-dual Weyl curvatures satisfy 
$$\int_M \frac{s^2}{24}d\mu \geq \int_M |W_+|^2 d\mu ,$$
with equality only if $g$ is K\"ahler-Einstein, of negative scalar curvature. 
\end{lem}

\begin{proof}
By the  first inequality of Theorem \ref{geometrize}, with $\ell=0$, we have 
$$
\frac{3}{4\pi^2}\int_M \frac{s^2}{24} d\mu \geq   (\corb)^2 (M, \Sigma , \beta ) ,
$$
with equality iff $g$ is K\"ahler-Einstein. 
On the other hand,  (\ref{gabo}--\ref{thom}) tells us that  we have
$$ (\corb )^2 (M, \Sigma , \beta ) =  \frac{1}{4\pi^2}\int_M \left(\frac{s^2}{24} + 2|W_+|^2\right)d\mu $$
for the orbifold Einstein metric $g$. Straightforward algebraic manipulation  then yields the desired  result. 
\end{proof}

This implies the following version of the generalized \cite{lno} Miyaoka-Yau inequality \cite{bes,yau}  for Einstein $4$-manifolds: 

\begin{thm} \label{sailor} 
For some integer $p\geq 2$, set $\beta=1/p$, and let 
 $(M,\Sigma, \beta)$ and $\mathfrak{c}$ be as in Lemma \ref{best}. 
If $b_+(M) > 1$,  suppose  that $SW_{\mathfrak{c}}(M,\Sigma,\beta) \neq 0$; if $b_+(M) = 1$, instead suppose that $\corb (L)$ is time-like  past-pointing, and that 
$SW^\downarrow_{\mathfrak{c}}(M,\Sigma,\beta)  \neq 0$.  If  $(M,\Sigma)$ admits an Einstein edge-cone metric  $g$ of 
cone angle $2\pi \beta$,  then
\begin{equation}
\label{richman}
(\chi - 3\tau )(M) \geq (1-\beta ) \left( \chi (\Sigma ) - (1+\beta ) [\Sigma ]^2 \right),
\end{equation}
   with equality iff  $[\Sigma]^2 = (p/2)\chi (\Sigma)$ and, up to constant rescaling,  $g$ is locally isometric to 
   the standard complex-hyperbolic metric on  $\CC\mathcal{H}_2= SU(2,1)/U(2)$.
   \end{thm}
\begin{proof} The Gauss-Bonnet-type formul{\ae} (\ref{gabo}--\ref{thom}) imply that 
\begin{eqnarray*}
(\chi - 3\tau )(M) &-& (1-\beta ) \Big( \chi (\Sigma ) - (1+\beta ) [\Sigma ]^2 \Big)
\\ &=&\frac{1}{8\pi^2}\int_M \left[\left( \frac{s^2}{24}-|W_+|^2\right) + 3|W_-|^2
-\frac{|\mathring{r}|^2}{2}
 \right] d\mu
\end{eqnarray*}
for any edge-cone metric $g$ on $(M,\Sigma)$ of any $\beta$. However, if the metric is Einstein and $\beta=1/p$ for some integer $p$, 
Proposition \ref{reggie} tells us that $g$ extends to $(M,\Sigma, \beta )$ as an Einstein orbifold metric, and Lemma \ref{soldier} 
then tells us that the right-hand-side is non-negative. The inequality thus follows. 

If the inequality is saturated, $W_-=0$, and Lemma \ref{soldier} moreover tells us that $g$ is a  K\"ahler-Einstein metric of negative scalar curvature. The $\End (\Lambda^+)$ block 
of the curvature tensor $\mathcal{R}$ in \eqref{curv}  is therefore a constant multiple of $\omega\otimes \omega$, and  the $\End (\Lambda^-)$ block is a constant multiple of the identity, while the off-diagonal blocks vanish.  Hence $\nabla \mathcal{R}=0$, and $g$ is locally symmetric. Since a globally symmetric spaces
is uniquely determined by the value of its curvature tensor at one point, it follows that $g$ is modeled on the complex-hyperbolic plane $\CC\mathcal{H}_2$,
up to a constant rescaling required to normalize the value of its scalar curvature. 
This done,  the  totally geodesic complex
curve $\Sigma$ then looks like $\CC\mathcal{H}_1\subset \CC\mathcal{H}_2$ in suitably chosen local uniformizing charts. However, in a local
orthonormal frame $e_1 , e_2 , e_ 3 , e_4$ for $\CC\mathcal{H}_2$ for which 
$e_2=Je_1$ and $e_4=Je_3$, one has $\mathcal{R}_{1234}=(1/2) \mathcal{R}_{1212}$,
so the curvature of the normal bundle of $\CC\mathcal{H}_1\subset \CC\mathcal{H}_2$ is half the curvature of its tangent bundle. 
Since this normal bundle  actually locally represents a $p^{\rm th}$ root of the normal bundle of $\Sigma$,  it follows that the normal bundle $N\to \Sigma$
satisfies $c_1(N) = (p/2) c_1(T^{1,0}\Sigma)$. In other words, we  have $[\Sigma]^2 = (p/2) \chi (\Sigma )$, as claimed. 
\end{proof}

When equality occurs, Theorem \ref{tailor} predicts that  
\begin{equation}
\label{poorman}
(\chi - 3\tau )(M) = -\frac{(p-1)^2}{2p} \chi (\Sigma).
\end{equation}
 It is therefore worth pointing out  that this  case does actually  occur. One  such example is provided by a beautiful construction \cite{polarchen,gost} 
 due  to Inoue. Let $\mathfrak{S}$ be a genus-$2$ complex curve on which 
 $\ZZ_5$ acts with $3$ fixed points, let $M\approx (\mathfrak{S}\times \mathfrak{S}) \# 3 \overline{\CP}_2$ be the blow-up of the product at the three corresponding
 points of the diagonal, and let $\Sigma \subset M$ be the disjoint union of the proper transforms of the graphs of the $5$ maps $\mathfrak{S}\to \mathfrak{S}$
 arising from the group action. Then $[\Sigma ]\in H^2(M,\ZZ)$ is divisible by $5$, and it follows that there is a $5$-fold cyclic
 branched cover  $Y\to M$ ramified at $\Sigma$. One checks that $Y$ saturates the Miyaoka-Yau inequality, and so is complex hyperbolic.
 The uniqueness of $\lambda < 0$ K\"ahler-Einstein metrics then predicts that the  complex hyperbolic metric is invariant under the 
 $\ZZ_5$-action, and so descends to   an orbifold Einstein metric on $(M,\Sigma, \beta )$, where $\beta = 1/5$. 
As a double-check of Theorem \ref{sailor}, notice that, in this example, we have 
 $$(\chi -3\tau)(M)= [(2-2\cdot 2)^2 + 3] - 3(-3)=16,$$
while  
 $$
  -\frac{(p-1)^2}{2p} \chi (\Sigma) =  -\frac{(5-1)^2}{2\cdot 5} 5 \chi (\mathfrak{S}) = 4^2 = 16.
 $$
 Thus \eqref{poorman} does indeed hold in this example, as predicted. 
 We leave it as an exercise for the  interested reader  to reprove the existence of the K\"ahler-Einstein edge metric $g$ on $(M,\Sigma, 1/5)$ 
 directly, using the latest techniques; 
 Theorem \ref{sailor} can then be used  to  give a different  proof of the fact that it is actually complex hyperbolic.  
 Of course,  infinitely many other examples exist that also that saturate \eqref{richman}; indeed, an infinite heirarchy of these 
  can be produced from the above example by simply  taking unbranched covers
 of $(\mathfrak{S}\times \mathfrak{S}) \# 3 \overline{\CP}_2$.
 
  \bigskip

 \bigskip 
 
 \noindent 
 {\bf Acknowledgements.} The author would  
like to thank  Xiuxiong Chen and Cliff Taubes  for encouragement  and 
 helpful suggestions.

  \end{document}